\documentclass[11pt]{article}
\usepackage[margin=1in]{geometry}
\usepackage[T1]{fontenc}
\usepackage{amsmath,amssymb,amsthm}
\usepackage{graphicx}
\graphicspath{{fig/}}
\usepackage{booktabs}
\usepackage{algorithm}
\usepackage{algpseudocode}
\usepackage{hyperref}
\hypersetup{hidelinks}

\DeclareMathOperator{\diag}{diag}

\DeclareMathOperator{\rank}{rank}

\DeclareMathOperator{\Span}{span}
\DeclareMathOperator{\trace}{trace}

\newcommand*{\trans}{^{\mathsf T}}
\newcommand*{\herm}{^{\mathsf H}}

\newcommand{\bmat}[1]{\begin{bmatrix}#1\end{bmatrix}}

\def\adots{\mathinner{\mkern2mu\raise1pt\hbox{.}\mkern2mu
    \raise4pt\hbox{.}\mkern2mu\raise7pt\hbox{.}\mkern1mu}}

\newcommand*{\proj}{\mathrm{p}}

\usepackage[normalem]{ulem}

\theoremstyle{plain}
\newtheorem{theorem}{Theorem}[section]
\newtheorem{proposition}[theorem]{Proposition}

\theoremstyle{definition}
\newtheorem{definition}[theorem]{Definition}

\title{A Structure-Preserving LOBPCG Method for Computing
Several of the Largest Generalized Singular Triplets}

\author{
Xinyu Shan
\thanks{
$^{1}$ Shanghai Institute for Mathematics and Interdisciplinary Sciences (SIMIS), 
Shanghai 200433, China;
$^{2}$ Research Institute of Intelligent Complex Systems, Fudan University, 
Shanghai 200433, China
(e-mail: \texttt{xyshan@simis.cn}).
}
}

\date{\today}

\begin{document}

\maketitle

\begin{abstract}
We develop a structure-preserving locally optimal block preconditioned conjugate gradient
(LOBPCG) method for computing several of the largest positive generalized singular values of a
matrix pair \((A,B)\), where \(B\) has full column rank, together with the associated
generalized singular vectors. 
The method is derived from the Hermitian-definite Jordan--Wielandt pencil and therefore
avoids explicitly forming the normal-equation matrices. 
By exploiting the block structure and spectral symmetry of this pencil, we
construct separate search subspaces for the left and right generalized singular vectors.
The resulting Rayleigh--Ritz procedure requires only a singular value decomposition of a
small projected cross matrix. 
We adapt the improved Hetmaniuk--Lehoucq (IHL) trick to construct stable conjugate search
directions in the two component spaces. 
We also establish an explicit correspondence between the coefficient matrices produced by the
componentwise and augmented Rayleigh--Ritz procedures. 
This correspondence permits the componentwise IHL construction to be used in an augmented, 
structure-preserving LOBPCG formulation without explicitly assembling the augmented basis. 
Numerical experiments demonstrate the computational advantages of the
proposed method over unstructured and Gram-matrix-based LOBPCG implementations.
\end{abstract}



\section{Introduction}
\label{sec:intro}
The generalized singular value decomposition (GSVD) was introduced by Van
Loan \cite{VanLoan1976} and subsequently developed by Paige and Saunders
\cite{PS1981} and others \cite{BZ1993,PE1993}. 
It appears in a wide range of numerical linear algebra and data analysis problems, 
including eigenvalue problems~\cite{Betcke2008}, least-squares problems~\cite{Hansen1989}, 
information retrieval~\cite{HJP2003}, and real-time signal processing~\cite{NNI2012}.

Several numerical algorithms based on dense matrix factorizations have been developed 
for computing the GSVD~\cite{BD1993,Paige1986}, and recent advances 
in the stable computation of the CS decomposition provide an alternative 
approach~\cite{Sutton2012}. 
For large-scale sparse matrix pairs, iterative methods are generally more suitable 
when only a few generalized singular triplets are required. 
Zha's algorithm~\cite{Zha1996}, which combines the CS decomposition with the 
Lanczos bidiagonalization process, computes a few extremal generalized singular
values and their associated vectors.
Jacobi--Davidson-type methods~\cite{Hochstenbach2009,HJ2023} can target 
generalized singular values at arbitrary locations and are particularly well 
suited for computing interior ones.
More recently, Liu et al.~\cite{LSS2026} proposed a contour-integral-based
algorithm that exploits the structure of the associated Jordan--Wielandt matrix pencil
to compute generalized singular values in a prescribed interval.
The LOBPCG method has also been applied to the GSVD through the normal-equation
generalized eigenvalue problem~\cite{DK2026}.

Originally proposed by Knyazev~\cite{Knyazev2001}, the LOBPCG method computes extremal
eigenpairs of symmetric-definite generalized eigenvalue problems.
At each iteration, the LOBPCG method performs a Rayleigh--Ritz procedure in the trial 
subspace spanned by the current approximations, preconditioned residuals, and previous
search directions.
Hetmaniuk and Lehoucq observed that a poorly chosen basis can make the Gram
matrix in the Rayleigh--Ritz procedure ill-conditioned, thereby slowing or even
preventing convergence \cite{HL2006}. They proposed a basis-updating strategy that
maintains orthogonality, commonly called the HL trick. An improved variant (IHL),
together with additional strategies for numerical robustness, was analyzed in
\cite{DSYG2018}.
The IHL trick has also been employed in a structure-preserving LOBPCG method for the
Bethe--Salpeter eigenvalue problem~\cite{ss2026}.

In this paper, we develop a structure-preserving LOBPCG method for computing several
of the largest nontrivial generalized singular values of a matrix pair \((A,B)\) and
their associated generalized singular vectors.
Unlike the existing LOBPCG method based on the normal-equation formulation~\cite{DK2026},
our method is derived from the generalized Hermitian-definite eigenvalue problem associated
with the Jordan--Wielandt matrix pencil~\cite{HJ2021}.
Instead of treating the resulting augmented eigenvalue problem as an unstructured
problem, we exploit its block structure and construct separate search subspaces for
the left and right generalized singular vectors.
Under a simplified Galerkin condition, the Rayleigh--Ritz procedure
reduces to computing a small SVD of the projected cross matrix.
To update the two search subspaces, we adapt the IHL trick to the componentwise
setting.
We further establish an explicit correspondence between the coefficient
representations arising from the simplified and structured Galerkin
conditions.
This correspondence enables the componentwise IHL construction to be incorporated
into the augmented structure-preserving LOBPCG framework.

The paper is organized as follows. 
In Section~\ref{sec:prelim}, we review the GSVD, the Jordan--Wielandt pencil, and the 
structured Galerkin condition.
Section~\ref{sec:method} develops the structure-preserving LOBPCG method. 
Numerical experiments are reported in Section~\ref{sec:numerics}.
Conclusions are drawn in Section~\ref{sec:concl}.

\section{Preliminaries}
\label{sec:prelim}
Throughout this paper, let \(A\in\mathbb{C}^{m\times n}\) and \(B\in\mathbb{C}^{p\times n}\),
with \(B\) of full column rank, i.e., \(\operatorname{rank}(B)=n\). 
We use \emph{nontrivial} to mean strictly positive and write \(q=\rank(A)>0\).
The derivation of the method uses \(k\) as the working block
size.
\(A\herm\) denotes the conjugate transpose of \(A\).
\(A\trans\) represents the transpose of \(A\).
\(I_d\) represents the identity matrix of order \(d\).
\(\Span(M)\) denotes the subspace spanned by the columns of \(M\).
Unless otherwise stated, orthogonality is Euclidean. 
The weighted inner product on the right component space is
\(\langle x,y\rangle_{B\herm B}=x\herm B\herm By\).

\subsection{Generalized singular value decomposition}
\begin{definition}[Generalized singular value decomposition]
\label{def:gsvd}
Let \(A\in \mathbb{C}^{m\times n}\) and \(B\in\mathbb{C}^{p\times n}\), and assume that \(\rank(B)=n\). 
Then there exist unitary matrices \(U\in\mathbb{C}^{m\times m}\) and \(V\in\mathbb{C}^{p\times p}\), together with a
nonsingular matrix \(X\in\mathbb{C}^{n\times n}\), such that
\[
U\herm AX=\Sigma_A, \qquad V\herm BX=\Sigma_B,
\]
where
\[
\Sigma_A =
\begin{bmatrix}
C & 0\\
0 & 0
\end{bmatrix},
\qquad
\Sigma_B =
\begin{bmatrix}
S & 0\\
0 & 0\\
0 & I_{n-q}
\end{bmatrix},
\]
\(C=\diag(\alpha_1,\ldots,\alpha_q)\), \(S=\diag(\beta_1,\ldots,\beta_q)\),
with \(\alpha_i>0\), \(\beta_i>0\), \(\alpha_i^2+\beta_i^2=1\), \(i=1,\ldots,q\).
The quantities
\[
\sigma_i=\frac{\alpha_i}{\beta_i},
\qquad i=1,\ldots,q,
\]
are called the \emph{nontrivial generalized singular values} of the pair \((A,B)\), ordered as
\(\sigma_1\ge\sigma_2\ge\cdots\ge\sigma_q>0\).
Let \(x_i\) denote the column of \(X\) associated with the pair \((\alpha_i,\beta_i)\), and define
\(w_i=x_i/\beta_i\).
Let \(u_i\) and \(v_i\) be the corresponding columns of \(U\) and \(V\).
Then the left and scaled right generalized singular vectors \(u_i\) and \(w_i\) satisfy
\begin{equation*}
Aw_i=\sigma_i u_i,
\qquad
A\herm u_i=\sigma_i B\herm Bw_i,
\qquad
u_i\herm u_j=\delta_{ij},
\qquad
w_i\herm B\herm Bw_j=\delta_{ij}.
\end{equation*}
The triple \((\sigma_i,u_i,w_i)\) is referred to as a \emph{generalized singular triplet} of the
pair \((A,B)\).
The normalized GSVD components are recovered as
\[
\beta_i=(1+\sigma_i^2)^{-1/2},\qquad
\alpha_i=\sigma_i\beta_i,\qquad
x_i=\beta_i w_i,\qquad v_i=Bw_i.
\]
\end{definition}

The nontrivial GSVD components of the matrix pair \((A,B)\) can be
characterized by two equivalent Hermitian generalized eigenvalue problems.
In particular, the right generalized singular vectors in the original GSVD scaling \(x_i\) satisfy
\begin{equation}
\label{eq:GEP1}
A\herm A x_i
=
B\herm B x_i \sigma_i^2.
\end{equation}
Alternatively, the generalized singular triplet \((\sigma_i,u_i,w_i)\) satisfies the augmented 
generalized eigenvalue problem induced by the Jordan--Wielandt pencil:
\begin{equation}
\label{eq:GEP2-A}
\check A
\bmat{u_i \\ w_i}
=
\check B
\bmat{u_i \\ w_i}\sigma_i,
\qquad i=1,\ldots,q,
\end{equation}
where
\begin{equation*}
\check A=
\begin{bmatrix}
0&A\\
A\herm &0
\end{bmatrix},
\qquad
\check B=
\begin{bmatrix}
I_m&0\\
0&B\herm B
\end{bmatrix}.
\end{equation*}
Compared with~\eqref{eq:GEP1}, formulation~\eqref{eq:GEP2-A} avoids
explicitly forming \(A\herm A\) and avoids squaring the generalized
singular values. 
It is therefore often preferable in terms of numerical stability, particularly
when \(A\) is ill-conditioned.

For a positive generalized singular value, the corresponding triplet
also satisfies the dual augmented relation
\begin{equation*}
\bmat{ & B \\ B\herm & }
\bmat{v_i \\ x_i/\alpha_i}
=
\bmat{I & \\ & A\herm A}
\bmat{v_i \\ x_i/\alpha_i}\sigma_i^{-1}.
\end{equation*}
This formulation requires \(\alpha_i>0\), ensuring that the scaling \(x_i/\alpha_i\) is well defined.
Moreover, the right-hand coefficient matrix \(\diag(I,A\herm A)\) is positive definite 
if and only if \(A\) has full column rank. 
Under this additional assumption, the dual relation defines a Hermitian-definite 
generalized eigenvalue problem and provides a useful alternative, particularly 
when \(B\herm B\) is ill-conditioned.

The method developed below is based exclusively on~\eqref{eq:GEP2-A}.
A key structured property of the augmented pencil is its spectral
symmetry about the origin: if \(\sigma_i\neq0\) is an eigenvalue with
eigenvector \([u_i\herm,w_i\herm]\herm\), then \(-\sigma_i\) is also an
eigenvalue, with the corresponding eigenvector \([u_i\herm,-w_i\herm]\herm\).
These paired eigenvalue relations can be written compactly as
\[
\bmat{ & A \\ A\herm & }
\bmat{u_i & u_i \\ w_i & -w_i}
=
\bmat{I & \\ & B\herm B}
\bmat{u_i & u_i \\ w_i & -w_i}
\bmat{\sigma_i & \\ & -\sigma_i},
\]
Thus, the nonzero eigenvalues occur in opposite-sign pairs
\(\pm\sigma_i\), with the associated eigenvectors exhibiting the
paired structure shown above.

\subsection{Trace variational characterization}
Let \(\sigma_1\ge \sigma_2\ge \cdots \ge \sigma_k\) denote the \(k\) largest 
generalized singular values of the matrix pair \((A,B)\).
Applying the Ky Fan trace variational principle to the associated Hermitian-definite
pencil yields
\begin{equation*}
\sum_{i=1}^k \sigma_i=\max_{Z\herm \check{B}Z=I_k}
\trace\bigl(Z\herm \check{A}Z\bigr).
\end{equation*}
The maximum is attained at
\[
Z=
\begin{bmatrix}
U_{\max}\\
W_{\max}
\end{bmatrix}
\in\mathbb{C}^{(m+n)\times k},
\]
where the columns of \(\sqrt{2}U_{\max}\) and \(\sqrt{2}W_{\max}\) are the left and
right generalized singular vectors associated with \(\Sigma_k=\diag(\sigma_1,\ldots,\sigma_k)\),
respectively.

Similarly,
\begin{equation*}
-\sum_{i=1}^k \sigma_i=\min_{Z\herm \check{B}Z=I_k}
\trace\bigl(Z\herm \check{A}Z\bigr).
\end{equation*}
The minimum is attained by
\[
Z=
\begin{bmatrix}
U_{\min}\\
W_{\min}
\end{bmatrix}
\in\mathbb{C}^{(m+n)\times k},
\]
whose columns are generalized eigenvectors associated with the \(k\) smallest eigenvalues
\(-\Sigma_k=\diag(-\sigma_1,\ldots,-\sigma_k)\).
The blocks \(\sqrt{2}U_{\min}\) and \(-\sqrt{2}W_{\min}\) correspond to the left and right
generalized singular vectors associated with the largest generalized singular
values \(\Sigma_k\).

Consequently, the partial GSVD of the matrix pair \((A,B)\) associated
with its \(k\) largest generalized singular values can be recovered
equivalently from either of the two extremal trace problems for the
associated Hermitian-definite pencil \((\check A,\check B)\). 
Specifically, maximizing the trace of \(Z\herm\check AZ\) over \(k\)-dimensional
\(\check B\)-orthonormal subspaces yields the positive eigenvalue branch,
whereas minimizing it yields the corresponding negative branch.

\subsection{LOBPCG method}
The LOBPCG method introduced by Knyazev~\cite{Knyazev2001} is a block
iterative eigensolver for computing a few eigenpairs associated with either the 
smallest or the largest eigenvalues of standard or generalized Hermitian eigenvalue
problems induced by a Hermitian-definite pencil \((A_1,A_2)\) (i.e., \(A_1\herm=A_1\in\mathbb C^{N\times N}\),
\(A_2\herm=A_2\in\mathbb C^{N\times N}\), and \(A_2\) is positive definite).
At iteration \(i\), the next approximation \(Z^{(i+1)}\) is sought in the \(3k\)-dimensional space
spanned by the current approximation \(Z^{(i)}\), the (preconditioned) residual \(F^{(i)}\), and the previous 
approximation \(Z^{(i-1)}\), i.e., in \(\Span(Z^{(i)},Z^{(i-1)},F^{(i)})\).
Assuming that \(V_{1,1}\) is nonsingular, replacing the previous approximation \(Z^{(i-1)}\) 
by the stabilized search direction \(P^{(i)}=Z^{(i)}-Z^{(i-1)}V_{1,1}\) preserves the search subspace 
unchanged, since \(\Span(P^{(i)}, Z^{(i)})=\Span(Z^{(i-1)}, Z^{(i)})\).
This reformulation improves the numerical stability of the orthogonalization as the iterates 
converge~\cite{DSYG2018,HL2006}.

A key practical issue is to construct \(P^{(i+1)}\) while maintaining a well-conditioned \(A_2\)-orthonormal 
basis without repeatedly orthogonalizing large vectors. 
The IHL trick~\cite{DSYG2018} achieves this by transferring the essential
orthogonalization to a small coefficient matrix.
For the full-rank case, assume that
\([Z^{(i)},Z_{\perp}^{(i)}]\) is an \(A_2\)-orthonormal search basis of
dimension \(3k\), with its first \(k\) columns equal to \(Z^{(i)}\).
The coefficient matrix produced by the Rayleigh--Ritz procedure is unitary and can be partitioned as
\[
V=
\bmat{
V_{1,1} & V_{1,2}\\
V_{2,1} & V_{2,2}
},
\qquad
V_{1,1}\in\mathbb C^{k\times k},
\quad
V_{1,2}\in\mathbb C^{k\times 2k},
\]
where \(V_{2,1}\in\mathbb C^{2k\times k}\) and \(V_{2,2}\in\mathbb C^{2k\times 2k}\).
Update \(Z^{(i+1)}\) and \(P^{(i+1)}\) as
\[
Z^{(i+1)}=\bigl[Z^{(i)},Z_{\perp}^{(i)}\bigr]\bmat{V_{1,1}\\V_{2,1}},\qquad
P^{(i+1)}=Z^{(i+1)}-Z^{(i)}V_{1,1}=\bigl[Z^{(i)},Z_{\perp}^{(i)}\bigr]\bmat{0\\V_{2,1}},
\]
where \(Z_{\perp}^{(i)}=[P^{(i)}, F^{(i)}]\).
To obtain a direction orthogonal to \(Z^{(i+1)}\), consider the orthogonal
projection of \(\bigl[0,V_{2,1}^{\herm}\bigr]^{\herm}\) onto the orthogonal
complement of \(\Span([V_{1,1}^{\herm},V_{2,1}^{\herm}]^{\herm})\):
\[
\begin{aligned}
\bmat{0\\V_{2,1}}-\bmat{V_{1,1}\\V_{2,1}}\bmat{V_{1,1}\\V_{2,1}}\herm\bmat{0\\V_{2,1}}
=\bmat{V_{1,2}\\V_{2,2}}\bmat{V_{1,2}\\V_{2,2}}\herm\bmat{0\\V_{2,1}}
=-\bmat{V_{1,2}\\V_{2,2}}V_{1,2}\herm V_{1,1}.
\end{aligned}
\]
Since the columns of \([V_{1,2}\herm, V_{2,2}\herm]\herm\) are already orthonormal, the 
remaining task is reduced to constructing an orthonormal basis for the column space of 
\(V_{1,2}\herm\), which is a small \(2k\times k\) matrix.
Omitting the right factor \(V_{1,1}\) does not alter the range of the projected history block
when \(V_{1,1}\) is nonsingular.
Assuming also that \(V_{1,2}\) has full row rank, let
\(V_{1,2}=LQ\herm\) be a thin LQ decomposition, where \(L\in\mathbb C^{k\times k}\) and
\(Q\in\mathbb C^{2k\times k}\) satisfies \(Q\herm Q=I_k\).
The IHL update then defines
\[
\bigl[Z^{(i+1)},P^{(i+1)}\bigr]
=
\bigl[Z^{(i)},Z_{\perp}^{(i)}\bigr]
V
\bmat{
I_k & \\
& Q
}.
\]
Hence, the IHL trick replaces the orthogonalization of large-scale vectors
by an orthogonalization involving only the small matrix \(V_{1,2}\), thereby
reducing the orthogonalization cost while improving the numerical robustness
of the LOBPCG method.

\subsection{Structured Galerkin condition}
\label{sec:galerkin}
We next recall the Galerkin condition associated with the augmented
Hermitian-definite formulation of the GSVD. Owing to the spectral symmetry
of the pencil \((\check A,\check B)\), an approximation to a positive
generalized singular value \(\sigma_i\) is naturally accompanied by an
approximation to its negative counterpart. Accordingly, let \([\tilde U\herm,\tilde W\herm]\herm\)
and \([\tilde U\herm,-\tilde W\herm]\herm\) with \(\tilde U\in\mathbb{C}^{m\times k}\),
\(\tilde W\in\mathbb{C}^{n\times k}\) be approximate eigenvector blocks associated with the approximate
eigenvalues \(\tilde\Sigma\) and \(-\tilde\Sigma\), respectively, where
\(\tilde\Sigma=\diag(\tilde\sigma_1,\ldots,\tilde\sigma_k)\) is positive diagonal.
A structure-preserving approximation is obtained by enforcing the Galerkin 
condition simultaneously for both eigenvalues in each symmetric pair, namely,
\begin{equation}
\label{eq:GSVD-Galerkin}
\Span\left(\check A\bmat{\tilde U & \tilde U \\ \tilde W & -\tilde W}
-\check B\bmat{\tilde U & \tilde U \\ \tilde W & -\tilde W}
\bmat{\tilde\Sigma \\ & -\tilde\Sigma}\right)
\perp\Span\left(\bmat{\tilde U & \tilde U \\ \tilde W & -\tilde W}\right).
\end{equation}
Following~\cite{LSS2026}, we refer to \eqref{eq:GSVD-Galerkin} as the
\emph{structured Galerkin condition}, which is equivalent to
\begin{equation*}
\mathbb{C}ases{
\Span(A\tilde W-\tilde U\tilde\Sigma)\perp\Span(\tilde U),\\
\Span(A\herm\tilde U-B\herm B\tilde W\tilde\Sigma)\perp\Span(\tilde W).}
\end{equation*}
Suppose further that the approximate left and right generalized singular
vectors satisfy \(\tilde U\herm\tilde U=I_k\), and \(\tilde W\herm B\herm B\tilde W=I_k\).
Then~\eqref{eq:GSVD-Galerkin} reduces to the \emph{simplified Galerkin condition} 
in~\cite{Hochstenbach2009}:
\begin{equation*}
\tilde U\herm A\tilde W=\tilde\Sigma.
\end{equation*}

We next describe how to construct approximations satisfying the simplified
Galerkin condition from prescribed left and right trial subspaces. 
Let \(S_U\in\mathbb{C}^{m\times 3k}\) and \(S_W\in\mathbb{C}^{n\times 3k}\) be bases for the search
subspaces, normalized such that \(S_U\herm S_U=I_{3k}\) and \(S_W\herm B\herm BS_W=I_{3k}\).
We seek approximate left and right generalized singular vectors of the form
\(\tilde U=S_UV_U\) and \(\tilde W=S_WV_W\), where \(V_U\), \(V_W\in\mathbb{C}^{3k\times k}\). 
Enforcing the residuals to be orthogonal to the corresponding trial subspaces gives
\[
S_U\herm(A\tilde W-\tilde U\tilde\Sigma)=0,\qquad
S_W\herm(A\herm\tilde U-B\herm B\tilde W\tilde\Sigma)=0.
\]
Defining the projected cross matrix \(M=S_U\herm AS_W\),
the above conditions reduce to
\[
MV_W=V_U\tilde\Sigma,
\qquad
M\herm V_U=V_W\tilde\Sigma.
\]
Thus, \(V_U\) and \(V_W\) are left and right singular vectors, and \(\tilde\Sigma\) 
contains the corresponding singular values of \(M\). 
Consequently, the desired approximations \(\tilde U=S_UV_U\) and \(\tilde W=S_WV_W\) are obtained by 
computing the selected singular triplets of the small projected matrix \(M\).

\section{A structure-preserving LOBPCG method}
\label{sec:method}
We now develop an efficient structure-preserving LOBPCG method for computing the
\(k\) largest nontrivial generalized singular values of the matrix pair \((A,B)\).
The method is motivated by the Hermitian-definite pencil \((\check A,\check B)\),
where \(\check B\) is positive definite because \(B\) has full column rank. Rather
than explicitly assembling augmented vectors, the method works with separate left
and right component spaces. Its central structural requirement is to preserve,
throughout the iteration, the pairing between the components associated with the
positive and negative eigenvalue branches.

The \(k\) largest positive approximate eigenpairs are denoted by
\(\bigl(\Theta,[U\herm,W\herm]\herm\bigr)\). 
Their negative counterparts, \(\bigl(-\Theta,[U\herm,-W\herm]\herm\bigr)\), 
represent the \(k\) smallest approximate eigenpairs.
The residuals for the eigenproblem \eqref{eq:GEP2-A} are
\begin{align*}
R_+ &=\check{A}
\begin{bmatrix} U \\ W \end{bmatrix}
-\check{B}\begin{bmatrix} U \\ W \end{bmatrix}\Theta
=\begin{bmatrix} R_U \\ R_W \end{bmatrix},
&
R_- &= \check{A}
\begin{bmatrix} U \\ -W \end{bmatrix}
+ \check{B}
\begin{bmatrix} U \\ -W \end{bmatrix}\Theta
=\begin{bmatrix} -R_U \\ R_W \end{bmatrix},
\end{align*}
where
\begin{equation*}
R_U=AW-U\Theta,
\qquad
R_W=A\herm U-B\herm BW\Theta.
\end{equation*}
The residuals \(R_{+}\) and \(R_{-}\) inherit the block symmetry of the paired
approximate eigenvectors, up to an irrelevant overall sign.
Let \(T_+\) and \(T_-\) be the preconditioners for \(R_+\) and \(R_-\), respectively.
Let
\[
\bmat{F_U \\  F_W}=T_+\bmat{R_U \\ R_W},\qquad
\bmat{F_U \\ -F_W}=T_-\bmat{-R_U \\ R_W},
\]
where \(T_-=-JT_+J\) with \(J=\diag(-I_m,I_n)\).
The preconditioned residuals still inherit the block structure.

\subsection{The IHL trick}
\label{sec:IHL}
At initialization, the history blocks are empty, \(P_U=P_W=[\,]\). Consequently,
the first Rayleigh--Ritz extraction uses the \(2k\)-column bases
\([U,F_U]\) and \([W,F_W]\). This first extraction constructs \(P_U\) and
\(P_W\); all subsequent extractions use the \(3k\)-column bases
\([U,P_U,F_U]\) and \([W,P_W,F_W]\). For clarity, the analysis below is written
for this steady-state \(3k\)-dimensional case. The first extraction follows from
the same argument after replacing the dimensions \(3k\) and \(2k\) of the full
basis and its complementary block by \(2k\) and \(k\), respectively.

Assume \(3k\le\min(m,n)\).
Let
\begin{equation*}
S_U=[U,P_U,F_U]\in\mathbb{C}^{m\times3k},\qquad
S_W=[W,P_W,F_W]\in\mathbb{C}^{n\times3k}
\end{equation*}
be bases for the search subspaces so that \(S_U\herm S_U=I_{3k}\) and
\(S_W\herm B\herm BS_W=I_{3k}\).
The simplified Galerkin condition leads to the following small SVD in the Rayleigh--Ritz procedure:
\begin{equation}
\label{eq:RR}
S_U\herm AS_W=U^\proj\Theta (W^\proj)\herm,  
\end{equation}
where \((U^\proj)\herm U^\proj=I_{3k}\) and \((W^\proj)\herm W^\proj=I_{3k}\), and 
\(\Theta=\diag(\theta_1,\ldots,\theta_{3k})\ge 0\).

At the next step, the updated basis matrices \([U,P_U]\) and \([W,P_W]\) are required
to satisfy the orthonormality conditions \([U,P_U]\herm [U,P_U]=I_{2k}\), and 
\([W,P_W]\herm B\herm B[W,P_W]=I_{2k}\).
These conditions could be enforced by explicitly orthogonalizing \(P_U\)
against \(U\) in the Euclidean inner product and \(P_W\) against \(W\) in
the \(B\herm B\)-inner product, followed by appropriate normalization.
Such operations, however, are performed on \(m\)- and \(n\)-dimensional
vectors and may therefore be computationally expensive.
We instead employ the IHL trick~\cite{HL2006,DSYG2018} to construct
\(P_U\) and \(P_W\) through low-dimensional coefficient transformations,
while preserving the required orthogonality of the corresponding
high-dimensional basis vectors.
In the present GSVD setting, the left and right coefficient transformations
must preserve the same number of columns and their columnwise correspondence,
so that the \(i\)th left and right components can be paired to form a
well-defined approximate eigenvector of the augmented problem.

Partition the unitary matrices \(U^\proj\), \(W^\proj\in\mathbb{C}^{3k\times3k}\) 
obtained from~\eqref{eq:RR} into \(2\times 2\) block form
\begin{equation*}
U^\proj=[U_1^\proj, U_2^\proj]=
\begin{bmatrix}
U^\proj_{1,1} & U^\proj_{1,2}\\
U^\proj_{2,1} & U^\proj_{2,2}
\end{bmatrix},
\qquad
W^\proj=[W_1^\proj, W_2^\proj]=
\begin{bmatrix}
W^\proj_{1,1} & W^\proj_{1,2}\\
W^\proj_{2,1} & W^\proj_{2,2}
\end{bmatrix}.
\end{equation*}
Here,
\[
U_1^\proj,\,W_1^\proj\in\mathbb{C}^{3k\times k},
\qquad
U_2^\proj,\,W_2^\proj\in\mathbb{C}^{3k\times2k},
\qquad
U_{1,2}^\proj,\,W_{1,2}^\proj\in\mathbb{C}^{k\times2k}.
\]
The first \(k\) columns \(U_1^\proj\) and \(W_1^\proj\) determine the updated approximations
\begin{align*}
U=S_UU_1^\proj,
\qquad
W=S_WW_1^\proj,  
\end{align*}
whereas \(U_2^\proj\) and \(W_2^\proj\) span their orthogonal
complements. 
The projected relationship is
\begin{equation*}
\begin{aligned}
\begin{bmatrix}
0\\U^\proj_{2,1}
\end{bmatrix}
-
\begin{bmatrix}
U^\proj_{1,1}\\U^\proj_{2,1}
\end{bmatrix}
\begin{bmatrix}
U^\proj_{1,1}\\U^\proj_{2,1}
\end{bmatrix}\herm
\begin{bmatrix}
0\\U^\proj_{2,1}
\end{bmatrix}
=
-U_{2}^\proj
(U^\proj_{1,2})\herm U^\proj_{1,1}.
\end{aligned}
\end{equation*}
Similarly,
\begin{equation*}
\begin{aligned}
\begin{bmatrix}
0\\W^\proj_{2,1}
\end{bmatrix}
-
\begin{bmatrix}
W^\proj_{1,1}\\W^\proj_{2,1}
\end{bmatrix}
\begin{bmatrix}
W^\proj_{1,1}\\W^\proj_{2,1}
\end{bmatrix}\herm
\begin{bmatrix}
0\\W^\proj_{2,1}
\end{bmatrix}
=
-W_{2}^\proj
(W^\proj_{1,2})\herm W^\proj_{1,1}.
\end{aligned}
\end{equation*}

If \(U^\proj_{1,1}\) is nonsingular, then
\[
\Span\bigl(U_2^\proj(U^\proj_{1,2})\herm U^\proj_{1,1}\bigr)
=\Span\bigl(U_2^\proj(U^\proj_{1,2})\herm\bigr).
\]
Thus the right factor \(U^\proj_{1,1}\) can be omitted when constructing
a basis for the conjugate direction.
Since \(U_2^\proj\) has orthonormal columns, only the small matrix \((U^\proj_{1,2})\herm\)
needs to be orthogonalized. The same argument applies to the right
component when \(W^\proj_{1,1}\) is nonsingular.
The necessary orthogonalization can therefore be carried out through small LQ factorizations of
\(U_{1,2}^\proj\) and \(W_{1,2}^\proj\).
Let \(U^\proj_{1,2}=L_UQ_U\herm\) and \(W^\proj_{1,2}=L_WQ_W\herm\) be the thin LQ factorizations, 
where \(L_U\), \(L_W\in\mathbb{C}^{k\times k}\) are nonsingular, and \(Q_U\), \(Q_W\in\mathbb{C}^{2k\times k}\) 
have orthonormal columns, i.e., \(Q_U\herm Q_U=I_k\), and \(Q_W\herm Q_W=I_k\).
Let
\[
Z_U=U_2^\proj Q_U,
\qquad
Z_W=W_2^\proj Q_W,
\]
The new \(P_U\) and \(P_W\) are then updated by
\begin{equation}
\label{eq:IHL-update}
P_U=S_UZ_U,
\qquad
P_W=S_WZ_W.
\end{equation}
The IHL update~\eqref{eq:IHL-update} preserves the required orthogonality and normalization conditions:
\(P_U\herm P_U=I_k\), \(P_U\herm U=0\), \(P_W\herm B\herm BP_W=I_k\), and \(P_W\herm B\herm B W=0\).
The range-preservation argument follows directly from the projection
identities above.

In finite precision, the paired blocks have equal numbers of columns, but the
independent orthogonalizations in the left and right component spaces may reveal
different numerical ranks. This mismatch can occur when \(F_U\) and \(F_W\) are
orthogonalized in their respective inner products and when
\((U_{1,2}^\proj)\herm\) and \((W_{1,2}^\proj)\herm\) are orthogonalized to form
\(Q_U\) and \(Q_W\). The approximation blocks \(U\) and \(W\), by contrast,
retain the same active column count by construction. We identify nearly dependent
directions using a prescribed rank-revealing tolerance and retain only the paired
directions whose column indices are accepted on both sides. The same common-index
rule is applied when constructing \(Q_U\) and \(Q_W\). Thus, coordinated deflation
preserves the one-to-one correspondence and equal column counts of all retained
left and right directions. Such deflation is a finite-precision safeguard and is
not commonly triggered in practice. To keep the notation and analysis concise,
we assume throughout the remainder of this section that no deflation occurs;
accordingly, all subsequent derivations use the nominal steady-state dimension
\(3k\). The implementation nevertheless applies the common-index rule above
whenever a rank deficiency is detected.
Algorithm~\ref{alg:gsvd} summarizes the LOBPCG method with the IHL trick.

\begin{algorithm}[htbp]
\caption{LOBPCG--GSVD solver}
\label{alg:gsvd}
\begin{algorithmic}[1]
\Require \(A\in\mathbb{C}^{m\times n}\), \(B\in\mathbb{C}^{p\times n}\) with
\(\rank(B)=n\); \(1\le l\le k\le\rank(A)\), \(3k\le\min(m,n)\);
initial \(U\), \(W\) with \(U\herm U=W\herm B\herm BW=I_k\);
a rule for applying \(T_+\); maximum iteration count \(N_{\max}\).
\Ensure The leading \(l\) approximate triplets.
\State Compute \(U\herm AW=U^\proj\Theta (W^\proj)\herm\), with descending singular values.
\State \(U\leftarrow UU^\proj\), \(W\leftarrow WW^\proj\), \(P_U\leftarrow[\,]\), \(P_W\leftarrow[\,]\).
\For{\(i=1,2,\ldots,N_{\max}\)}
  \State \(R_U\leftarrow AW-U\Theta\), \(R_W\leftarrow A\herm U-B\herm BW\Theta\).
  \If{the leading \(l\) approximate triplets satisfy the stopping criterion}
    \State \Return \(\Theta(1:l,1:l)\), \(U(:,1:l)\), \(W(:,1:l)\).
  \EndIf
  \State Apply \([F_U\herm,F_W\herm]\herm
         \leftarrow T_+[R_U\herm,R_W\herm]\herm\).
  \State Orthogonalize \(F_U\) against \([U,P_U]\), and
         \(B\herm B\)-orthogonalize \(F_W\) against \([W,P_W]\).
  \State Remove unmatched columns from \(F_U\) and \(F_W\) to preserve their columnwise correspondence.
  \State Form \(S_U=[U,P_U,F_U]\), \(S_W=[W,P_W,F_W]\).
  \State Compute the SVD: \(S_U\herm AS_W=U^\proj\Theta(W^\proj)\herm\), with descending singular values.
  \State Apply rank-revealing orthogonalization to
         \((U^\proj_{1,2})\herm\) and \((W^\proj_{1,2})\herm\).
  \State Form \(Q_U\), \(Q_W\) using the common accepted indices.
  \State \(U\leftarrow S_UU^\proj_1\), \(W\leftarrow S_WW^\proj_1\), \(\Theta\leftarrow\Theta(1:k,1:k)\).
  \State \(Z_U=U_2^\proj Q_U\), \(Z_W=W_2^\proj Q_W\).
  \State \(P_U\leftarrow S_UZ_U\), \(P_W\leftarrow S_WZ_W\).
\EndFor
\State \Return \(\Theta(1:l,1:l)\), \(U(:,1:l)\), \(W(:,1:l)\).
\end{algorithmic}
\end{algorithm}

\subsection{Coefficient correspondence between two Rayleigh--Ritz procedures}
Assume that
\(S_U\herm S_U=I_{3k}\), and \(S_W\herm B\herm BS_W=I_{3k}\).
Then
\begin{equation*}
S=\frac{1}{\sqrt{2}}\begin{bmatrix} 
S_U & S_U \\ S_W & -S_W 
\end{bmatrix}
\in\mathbb{C}^{(m+n)\times 6k}
\end{equation*}
is \(\check B\)-orthonormal, i.e., \(S\herm\check BS=I_{6k}\).
The structured Galerkin condition with respect to the augmented search subspace \(\Span(S)\) leads 
to the reduced eigenvalue problem
\begin{equation}
\label{eq:scheme1}
S\herm\check ASV_+=V_+\Theta,
\end{equation}
where \(\Theta\in\mathbb{R}^{3k\times3k}\) is diagonal with positive
diagonal entries, and \(V_+\in\mathbb{C}^{6k\times3k}\) has orthonormal
columns \(V_+\herm V_+=I_{3k}\).
The corresponding approximate eigenvector block of
\((\check A,\check B)\) is
\(X_+=SV_+\).

The IHL trick in Section~\ref{sec:IHL} is formulated in terms of
\(U^\proj\) and \(W^\proj\), the coefficient matrices associated with the
basis matrices \(S_U\) and \(S_W\). 
Their leading columns update \(U\) and \(W\), while their remaining columns 
are used to construct \(P_U\) and \(P_W\). 
In contrast, the augmented projection~\eqref{eq:scheme1} directly yields
only \(V_+\), the coefficient matrix of \(X_+=SV_+\) with respect to the
augmented basis \(S\).
This motivates recovering \(U^\proj\) and \(W^\proj\) from \(V_+\); the
required relationship is established in the following proposition.

\begin{proposition}
\label{prop:equiv}
Let \(S_U\), \(S_W\), \(S\) be as above, and suppose that
\(M=S_U\herm AS_W\) is nonsingular. Let
\(V_+=[(V_+^{\mathrm{top}})\herm,(V_+^{\mathrm{bot}})\herm]\herm
\in\mathbb{C}^{6k\times3k}\) have orthonormal columns satisfying
\(S\herm\check ASV_+=V_+\Theta\), where \(\Theta\) is positive diagonal,
\(V_+^{\mathrm{top}}\) and \(V_+^{\mathrm{bot}}\) have equal size. 
Then the matrices
\begin{equation}
\label{eq:recover-component-coeff}
U^\proj=V_+^{\mathrm{top}}+V_+^{\mathrm{bot}},\qquad
W^\proj=V_+^{\mathrm{top}}-V_+^{\mathrm{bot}}
\end{equation}
are unitary matrices satisfying
\begin{equation}\label{eq:recovered-projected-svd}
M=U^\proj\Theta(W^\proj)\herm.
\end{equation}
Furthermore, define \(V_-=[(V_+^{\mathrm{bot}})\herm,(V_+^{\mathrm{top}})\herm]\herm\).
Then \(S\herm\check ASV_-=-V_-\Theta\), \(V_-\herm V_-=I_{3k}\) and \(V_+\herm V_-=0\).
Consequently,
\begin{equation*}
[X_+,X_-]=[SV_+,SV_-]=\frac1{\sqrt2}
\begin{bmatrix}
S_UU^\proj & S_UU^\proj \\ 
S_WW^\proj & -S_WW^\proj
\end{bmatrix}.
\end{equation*}
\end{proposition}

\begin{proof}
Let
\[
H=\frac{1}{\sqrt{2}}
\begin{bmatrix}
I_{3k}&I_{3k}\\
I_{3k}&-I_{3k}
\end{bmatrix},
\qquad
D=
\begin{bmatrix}
0&M\\
M\herm&0
\end{bmatrix}.
\]
Then \(H\herm=H\), \(H^2=I_{6k}\), and
\(S\herm\check AS=HDH\).
By the definitions of \(U^\proj\) and \(W^\proj\), applying \(H\) to \(V_+\) gives
\[
HV_+
=
\frac{1}{\sqrt{2}}
\begin{bmatrix}
U^\proj\\
W^\proj
\end{bmatrix}.
\]
Multiplying \(S\herm\check ASV_+=V_+\Theta\) from the left by \(H\)
therefore gives
\[
D
\begin{bmatrix}
U^\proj\\
W^\proj
\end{bmatrix}
=
\begin{bmatrix}
U^\proj\\
W^\proj
\end{bmatrix}\Theta.
\]
Consequently, \(MW^\proj=U^\proj\Theta\), and
\(M\herm U^\proj=W^\proj\Theta\).

Set \(G_U=(U^\proj)\herm U^\proj\), \(G_W=(W^\proj)\herm W^\proj\).
Since \(V_+\herm V_+=I_{3k}\), we have
\(G_U+G_W=2I_{3k}\).
Moreover, the preceding singular-vector relations imply
\(G_U\Theta=\Theta G_W\).
Eliminating \(G_W\) gives
\[
(G_U-I_{3k})\Theta
+\Theta(G_U-I_{3k})=0.
\]
Since \(\Theta\) is positive diagonal, this homogeneous Sylvester
equation has only the trivial solution.
Hence \(G_U=G_W=I_{3k}\).
Thus, \(U^\proj\) and \(W^\proj\) are unitary. It follows from
\(MW^\proj=U^\proj\Theta\) that
\(M=U^\proj\Theta(W^\proj)\herm\),
which proves~\eqref{eq:recovered-projected-svd}.
According to the definition of \(V_-\), we obtain
\[
HV_-=
\frac1{\sqrt2}
\begin{bmatrix}
U^\proj\\
-W^\proj
\end{bmatrix}.
\]
Using the two singular-vector relations above, we obtain
\[
D
\begin{bmatrix}
U^\proj\\
-W^\proj
\end{bmatrix}
=
-\begin{bmatrix}
U^\proj\\
-W^\proj
\end{bmatrix}\Theta.
\]
Therefore, \(S\herm\check ASV_-=HDHV_-=-V_-\Theta\).
Furthermore,
\[
V_-\herm V_-=I_{3k},
\qquad
V_+\herm V_-
=\frac12(G_U-G_W)=0.
\]

Finally, since
\[
S=
\begin{bmatrix}
S_U&0\\
0&S_W
\end{bmatrix}H,
\]
direct substitution gives
\[
SV_+
=\frac1{\sqrt2}
\begin{bmatrix}
S_UU^\proj\\
S_WW^\proj
\end{bmatrix},
\qquad
SV_-
=\frac1{\sqrt2}
\begin{bmatrix}
S_UU^\proj\\
-S_WW^\proj
\end{bmatrix}.
\]
Hence,
\[
[SV_+,SV_-]
=
\frac1{\sqrt2}
\begin{bmatrix}
S_UU^\proj&S_UU^\proj\\
S_WW^\proj&-S_WW^\proj
\end{bmatrix}.\qedhere
\]
\end{proof}

\subsection{The IHL trick from the structured Galerkin condition}
\label{sec:aug-IHL}
Algorithm~\ref{alg:gsvd} constructs the IHL coefficients from the projected
SVD, where \(U^\proj\) and \(W^\proj\) are directly available. 
In the augmented Rayleigh--Ritz procedure, however, the projected eigenproblem
provides \(V_+\) and \(\Theta\) instead. 
Proposition~\ref{prop:equiv} allows us to recover \(U^\proj\) and \(W^\proj\) from \(V_+\).
The partition and orthogonalization described in Section~\ref{sec:IHL} can then be applied
to obtain
\[
Z_U=U_2^\proj Q_U,\qquad Z_W=W_2^\proj Q_W.
\]
Thus, the IHL update for the augmented procedure can be constructed through
\(P_U=S_UZ_U\), \(P_W=S_WZ_W\), \(U=S_UU_1^\proj\), and \(W=S_WW_1^\proj\).

Under the full-row-rank assumptions of Section~\ref{sec:IHL}, define
the paired augmented coefficient matrix
\begin{equation*}
C_P=\frac12
\begin{bmatrix}
Z_U+Z_W & Z_U-Z_W\\
Z_U-Z_W & Z_U+Z_W
\end{bmatrix}\in\mathbb{C}^{6k\times2k}.
\end{equation*}
The matrix \(P\) can be computed by
\begin{equation*}
P=SC_P
=\frac1{\sqrt2}\begin{bmatrix}
P_U&P_U\\P_W&-P_W
\end{bmatrix}.
\end{equation*}
Thus the augmented conjugate direction can be represented and updated through its
two components; neither \(S\) nor the large matrix \(P\)
needs to be explicitly assembled.

To describe the orthogonality and history-preserving properties of this
construction, define
\[
X_{\mathrm{old}}=\frac1{\sqrt2}
\bmat{U_{\mathrm{old}} & U_{\mathrm{old}} \\ W_{\mathrm{old}} & -W_{\mathrm{old}}},\qquad
X=\frac1{\sqrt2}\bmat{U & U \\W & -W}.
\]
Here, \(U_{\mathrm{old}}\) and \(W_{\mathrm{old}}\) denote the previous approximations stored 
in the first \(k\) columns of \(S_U\) and \(S_W\), respectively.
The componentwise orthogonality relations satisfied by \(U\), \(W\), \(P_U\), 
and \(P_W\) yield
\begin{equation*}
P\herm\check B P=I_{2k},\qquad
P\herm\check B X=0.
\end{equation*}
Furthermore, direct calculation gives
\(
XX\herm\check B
=\diag\bigl(UU\herm,
WW\herm B\herm B\bigr).
\)
Set
\[
Y=(I_{m+n}-XX\herm\check B)X_{\mathrm{old}}.
\]
Since \(X\herm\check B X=I_{2k}\),
the matrix \(XX\herm\check B\) is the \(\check B\)-orthogonal
projector onto \(\Span(X)\).  Define
\[
Y_U=(I_m-UU\herm)U_{\mathrm{old}},
\qquad
Y_W=(I_n-WW\herm B\herm B)W_{\mathrm{old}}.
\]
Then
\[
Y=\frac1{\sqrt2}
\begin{bmatrix}Y_U&Y_U\\Y_W&-Y_W\end{bmatrix}.
\]
To verify the componentwise range preservation, let
\(E=[I_k,0_{k\times2k}]\herm\), so that \(U_{\mathrm{old}}=S_UE\) and \(W_{\mathrm{old}}=S_WE\).
Using the unitarity of \([U_1^\proj,U_2^\proj]\) and the thin LQ
factorization \(U_{1,2}^\proj=L_UQ_U\herm\), we obtain
\[
\begin{aligned}
Y_U
=S_U\bigl(I_{3k}-U_1^\proj(U_1^\proj)\herm\bigr)E
=S_UU_2^\proj(U_{1,2}^\proj)\herm
=P_UL_U\herm.
\end{aligned}
\]
Likewise, \(Y_W=P_WL_W\herm\).
Since \(L_U\) and \(L_W\) are nonsingular, it follows that
\[
\Span(P_U)=\Span(Y_U),
\qquad
\Span(P_W)=\Span(Y_W).
\]
Moreover, for
\[
H_k=\frac1{\sqrt2}\begin{bmatrix}I_k&I_k\\I_k&-I_k\end{bmatrix},
\]
which is unitary, we have
\[
PH_k=\diag(P_U,P_W),
\qquad
YH_k=\diag(Y_U,Y_W).
\]
Right multiplication by \(H_k\) does not change the column space.
The two componentwise range equalities therefore imply
\(\Span(P)=\Span(Y)\).
It remains to compare the two enlarged subspaces.  Let
\(C=X\herm\check B X_{\mathrm{old}}\).  
From the definition of \(Y\), we get
\(Y=X_{\mathrm{old}}-XC\),
which shows that \(\Span(X,Y)=
\Span(X,X_{\mathrm{old}})\).  
Therefore, using \(\Span(P)=\Span(Y)\), we obtain
\(\Span(X,X_{\mathrm{old}})=\Span(X,P)\).
Thus, with \(X\) retained in the search space, replacing
\(X_{\mathrm{old}}\) with \(P\) leaves the augmented
search subspace unchanged.

\section{Numerical experiments}
\label{sec:numerics}
All experiments were performed in MATLAB R2022b on a Linux server with two
16-core Intel Xeon Gold 6226R processors at 2.90 GHz and 1024 GB of main
memory. 
All reported CPU times are total runtimes and include every stage of the
corresponding algorithm.
We use the twelve sparse matrices \(A\) from the SuiteSparse
Matrix Collection (formerly the University of Florida Sparse Matrix
Collection) \cite{DH2011}; the matrix names are
listed in Table~\ref{tab:matrices}.
Here \(\operatorname{nnz}(A)\) denotes the number of nonzero entries in
\(A\), and a superscript \(\trans\) on a matrix name indicates that the
transpose of the collection matrix is used.
In the GSVD experiments, the constraint matrices are
\begin{equation*}
B_1 =
\begin{bmatrix}
5 & 1 & & \\
1 & 5 & \ddots & \\
& \ddots & \ddots & 1 \\
& & 1 & 5
\end{bmatrix}
\in\mathbb{C}^{n\times n}, \qquad
B_2 =
\begin{bmatrix}
1 & -1 & & & \\
& 1 & -1 & & \\
& & \ddots & \ddots & \\
& & & 1 & -1
\end{bmatrix}^{\!\top}
\in\mathbb{C}^{(n+1)\times n},
\end{equation*}
i.e.\ \(B_1\) is the tridiagonal matrix with diagonal \(5\) and off-diagonal
\(1\), and \(B_2\) is the transpose of the first-order difference operator on 
\(n+1\) grid values, mapping \(\mathbb R^{n+1}\) to \(\mathbb R^n\) before transposition. 
Both \(B_1\) and \(B_2\) have full column rank.

\begin{table}[htbp]
\setlength{\tabcolsep}{3pt}
\centering
\caption{List of test matrices.}
\label{tab:matrices}
\begin{tabular}{lccccc}
\hline
ID & Matrix \(A\) & \(m\) & \(n\) & \(nnz(A)\) & \(\lVert A\rVert_2\)\\
\hline
 1 & \texttt{plat1919}               & \phantom{00}1,919 & \phantom{0}1,919 & \phantom{0}32,399
                                     & \phantom{00}2.93  \\
 2 & \texttt{rosen10}\(\trans\)      & \phantom{00}6,152 & \phantom{0}2,056 & \phantom{0}68,302
                                     & \phantom{}20,200   \\
 3 & \texttt{GL7d12}                 & \phantom{00}8,899 & \phantom{0}1,019 & \phantom{0}37,519
                                     & \phantom{00}14.4   \\
 4 & \texttt{3elt\_dual}             & \phantom{00}9,000 & \phantom{0}9,000 & \phantom{0}26,556
                                     & \phantom{0000}3    \\
 5 & \texttt{fv1}                    & \phantom{00}9,604 & \phantom{0}9,604 & \phantom{0}85,264
                                     & \phantom{00}4.52   \\
 6 & \texttt{shuttle\_eddy}          & \phantom{0}10,429 & \phantom{}10,429 & \phantom{}103,599
                                     & \phantom{00}16.2   \\
 7 & \texttt{nopoly}                 & \phantom{0}10,774 & \phantom{}10,774 & \phantom{0}70,842
                                     & \phantom{00}23.3   \\
 8 & \texttt{flower\_5\_4}\(\trans\) & \phantom{0}14,721  & \phantom{0}5,226& \phantom{0}43,942
                                     & \phantom{00}5.53   \\
 9 & \texttt{barth5}                 & \phantom{0}15,606 & \phantom{}15,606 & \phantom{0}61,484
                                     & \phantom{00}4.23  \\
10 & \texttt{L-9}                    & \phantom{0}17,983 & \phantom{}17,983 & \phantom{0}71,192
                                     & \phantom{0000}4    \\
11 & \texttt{crack\_dual}            & \phantom{0}20,141 & \phantom{}20,141 & \phantom{0}60,086
                                     & \phantom{0000}3   \\
12 & \texttt{rel8}                   & \phantom{}345,688 & \phantom{}12,347 & \phantom{}821,839
                                     & \phantom{00}18.3   \\
\hline
\end{tabular}
\end{table}

Unless otherwise stated, we seek the \(l=35\) largest generalized singular
values and the corresponding generalized singular vectors, and the block size is set to
\(k=\lceil 3l/2\rceil\).
The initial blocks \(U^{(0)}\in\mathbb{C}^{m\times k}\),  \(W^{(0)}\in\mathbb{C}^{n\times k}\) 
are randomly generated.
A triplet \((\hat\sigma_i,\hat u_i,\hat w_i)\) is considered converged when
\begin{equation*}
\begin{aligned}
\bigl\|A\hat w_i - \hat u_i\hat\sigma_i\bigr\|_2
&\le \mathrm{tol}\,\bigl(\|A\|_2\,\|\hat w_i\|_2 + |\hat\sigma_i|\bigr),\\
\bigl\|A\herm \hat u_i - B\herm B\,\hat w_i\hat\sigma_i\bigr\|_2
&\le \mathrm{tol}\,\bigl(\|A\|_2 + |\hat\sigma_i|\,\|B\|_2^2\,\|\hat w_i\|_2\bigr).
\end{aligned}
\end{equation*}
Unless stated otherwise, we set \(\mathrm{tol}=10^{-14}\sqrt{m}\).
Equivalently, the iteration stops when the largest of the two scaled
residuals over the leading \(l\) approximate triplets satisfies
\begin{equation*}
\max_{1\le i\le l}
\max\!\Bigl\{
\frac{\|A\hat w_i-\hat u_i\hat\sigma_i\|_2}{\|A\|_2\|\hat w_i\|_2+|\hat\sigma_i|},
\frac{\|A\herm \hat u_i - B\herm B\hat w_i\hat\sigma_i\|_2}{\|A\|_2+|\hat\sigma_i|\,\|B\|_2^2\|\hat w_i\|_2}
\Bigr\}
\;\le\; 10^{-14}\sqrt{m},
\end{equation*}
or when the iteration count reaches \(3000\).
The spectral norms used in the stopping criteria are estimated in MATLAB as
\texttt{normA = normest(A,1e-3)} and \texttt{normB = normest(B,1e-3)}.
In all convergence-history plots, the vertical axis shows the maximum of the
scaled residuals over the leading \(l\) approximate triplets at each iteration.
The relations \(U\herm U=I\) and \(W\herm B\herm BW=I\) are maintained by the
orthogonalization and IHL updates in Algorithm~\ref{alg:gsvd}, up to roundoff and
the tolerance used in the rank-revealing orthogonalization.

\subsection{Unpreconditioned GSVD computation with \(B=B_1\)}
\label{sec:gsvdb1}
We compare Algorithm~\ref{alg:gsvd} with two augmented implementations.
The structured augmented variant, \texttt{LOBPCG-eig}, uses the structured Galerkin 
condition \eqref{eq:GSVD-Galerkin}.
It recovers the component coefficients by~\eqref{eq:recover-component-coeff}
and applies the IHL construction in Section~\ref{sec:aug-IHL}.
The third implementation, \texttt{LOBPCG}, is an augmented eigensolver
that does not exploit the block structure of \((\check A,\check B)\).
The search subspaces of all three algorithms contain information about the
\(k\) largest and the \(k\) smallest approximate generalized eigenpairs.
All algorithms use the same initial guess.

We compute the \(35\) largest generalized singular triplets of \((A,B_1)\)
without preconditioning. 
Figure~\ref{fig:gsvdB1} presents the convergence histories, while Table~\ref{tab:gsvdB1}
reports the corresponding CPU time.
For the matrices in Table~\ref{tab:matrices}, all tests reach the prescribed
tolerance within the iteration limit except those for \((\texttt{fv1},B_1)\).
Reference generalized eigenvalues computed with MATLAB's \texttt{eigs}
function for the pencil \((A\herm A,B\herm B)\) indicate that the gap between
the \(35\)-th and the \(54\)-th largest generalized singular values is
\(7.3889\times 10^{-5}\), 
while Algorithm~\ref{alg:gsvd} reduces the maximum scaled residual below
\(10^{-7}\) at iteration \(2136\).

Table~\ref{tab:gsvdB1} lists the CPU times of the three algorithms.
Benefiting from the IHL trick and the efficient exploitation of the structured 
Galerkin condition, Algorithm~\ref{alg:gsvd} achieves a substantial CPU time advantage.
In contrast, the standard LOBPCG method, which does not exploit the underlying
matrix structure, incurs the highest computational cost.

\begin{figure}[htbp]
\centering
\begin{tabular}{ccc}
\includegraphics[width=0.3\textwidth]{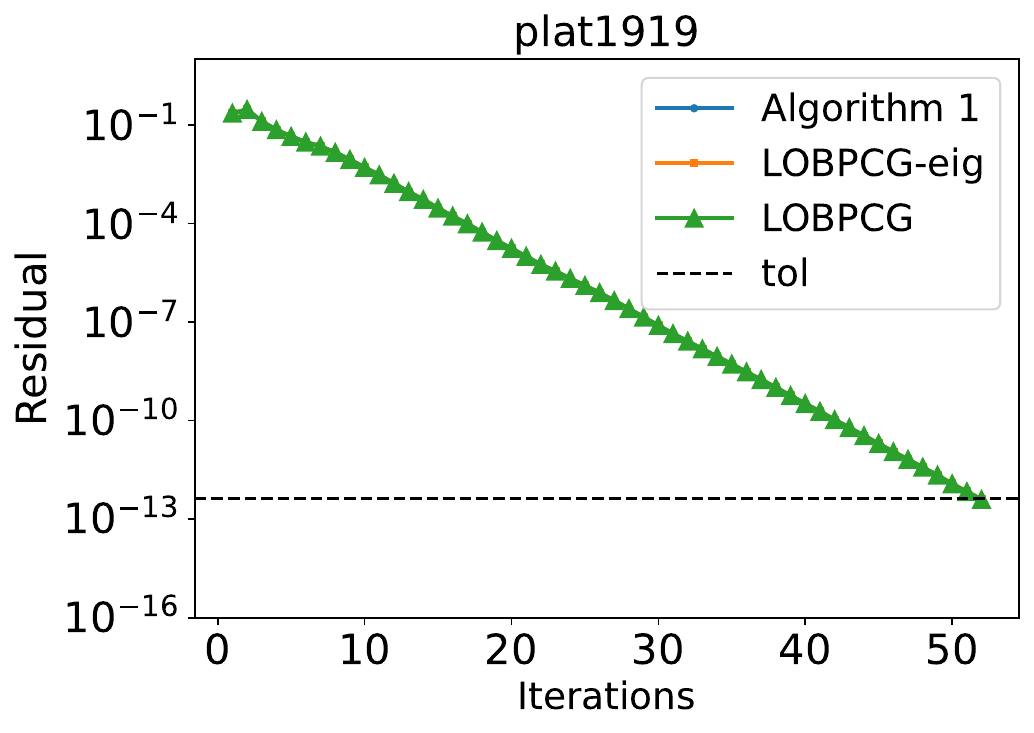} &
\includegraphics[width=0.3\textwidth]{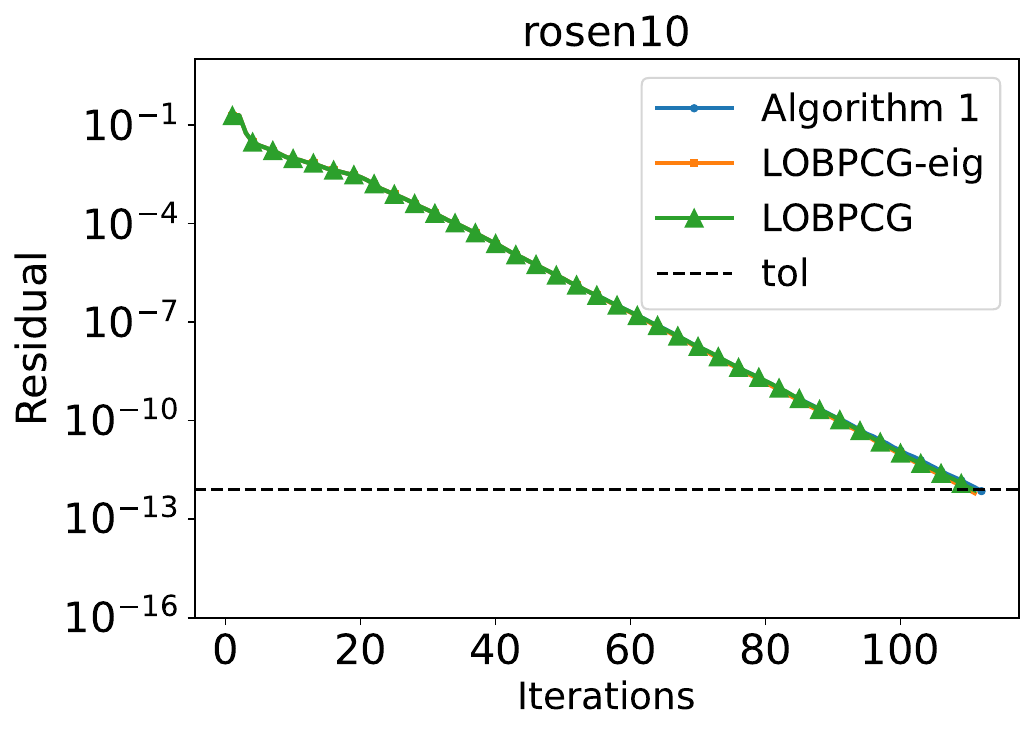} &
\includegraphics[width=0.3\textwidth]{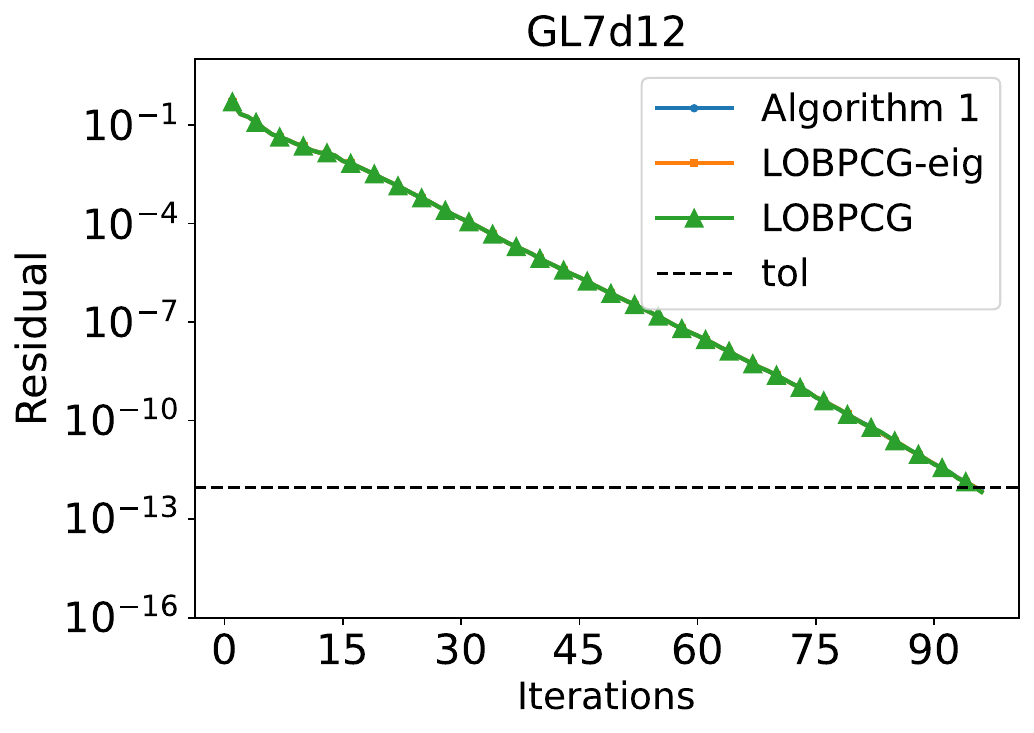} \\
\includegraphics[width=0.3\textwidth]{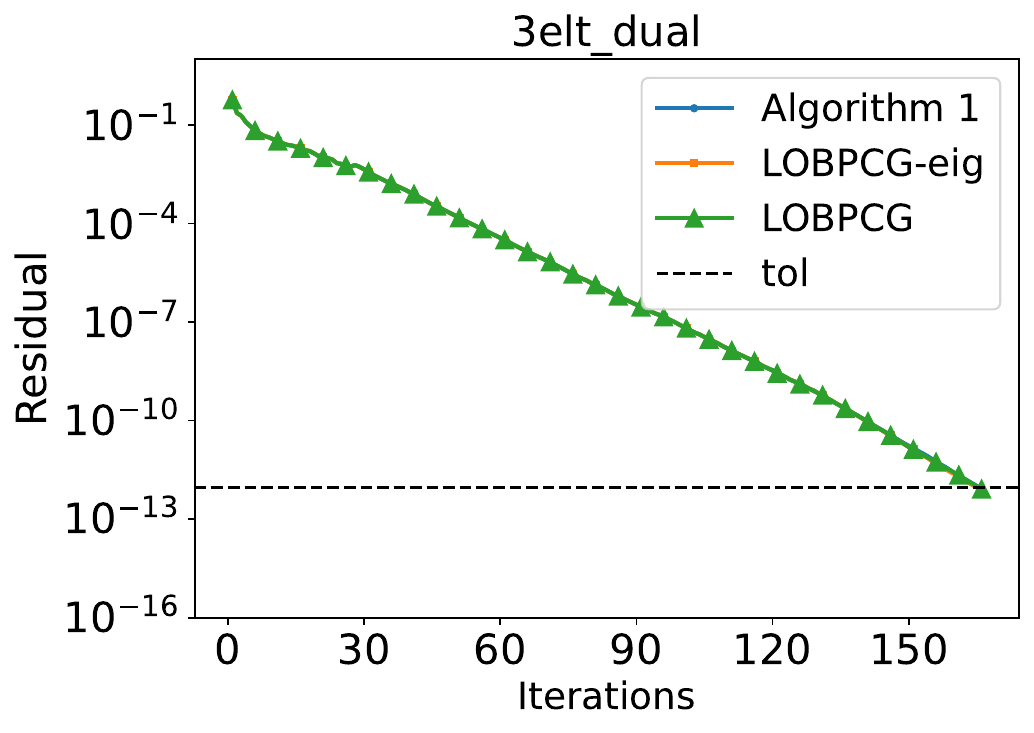} &
\includegraphics[width=0.3\textwidth]{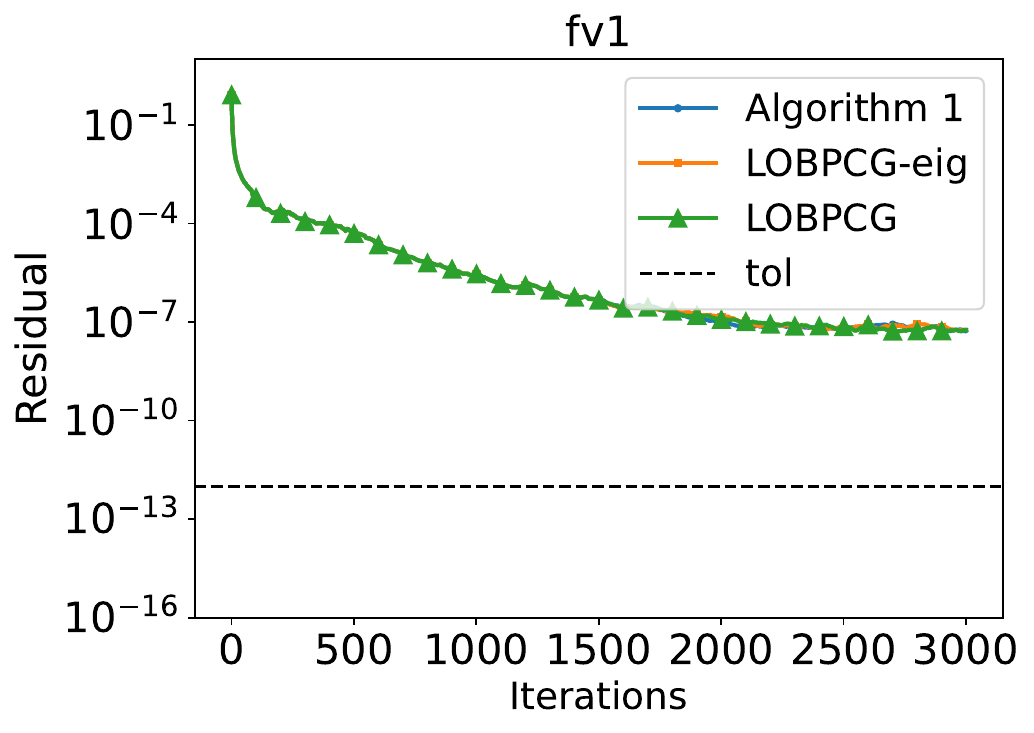} &
\includegraphics[width=0.3\textwidth]{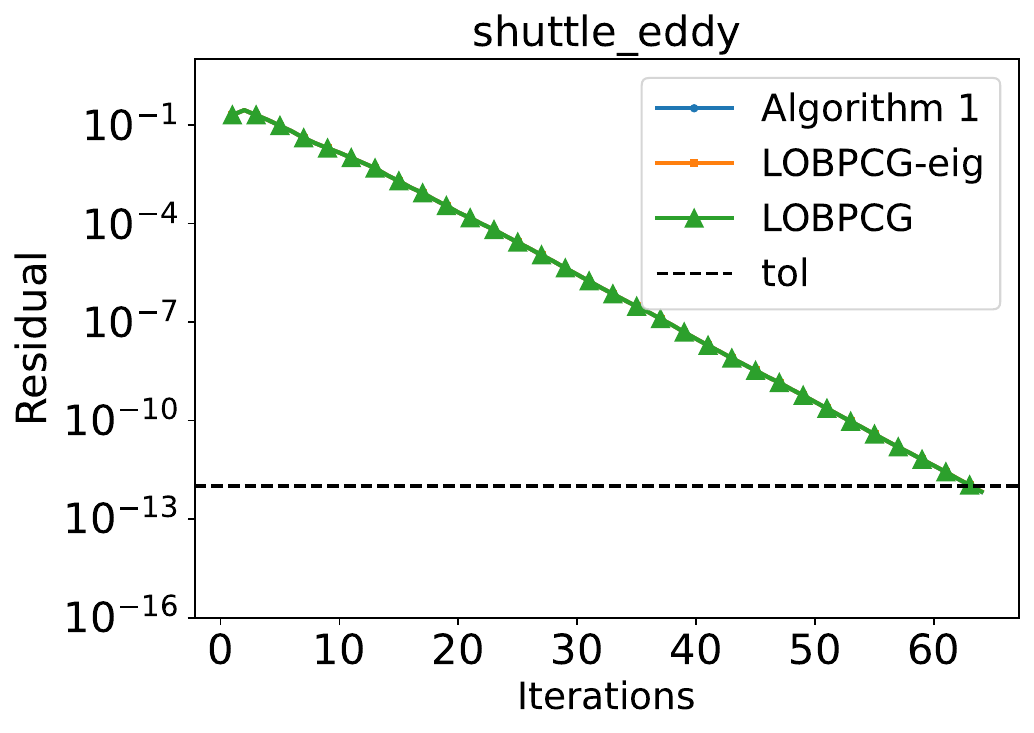} \\
\includegraphics[width=0.3\textwidth]{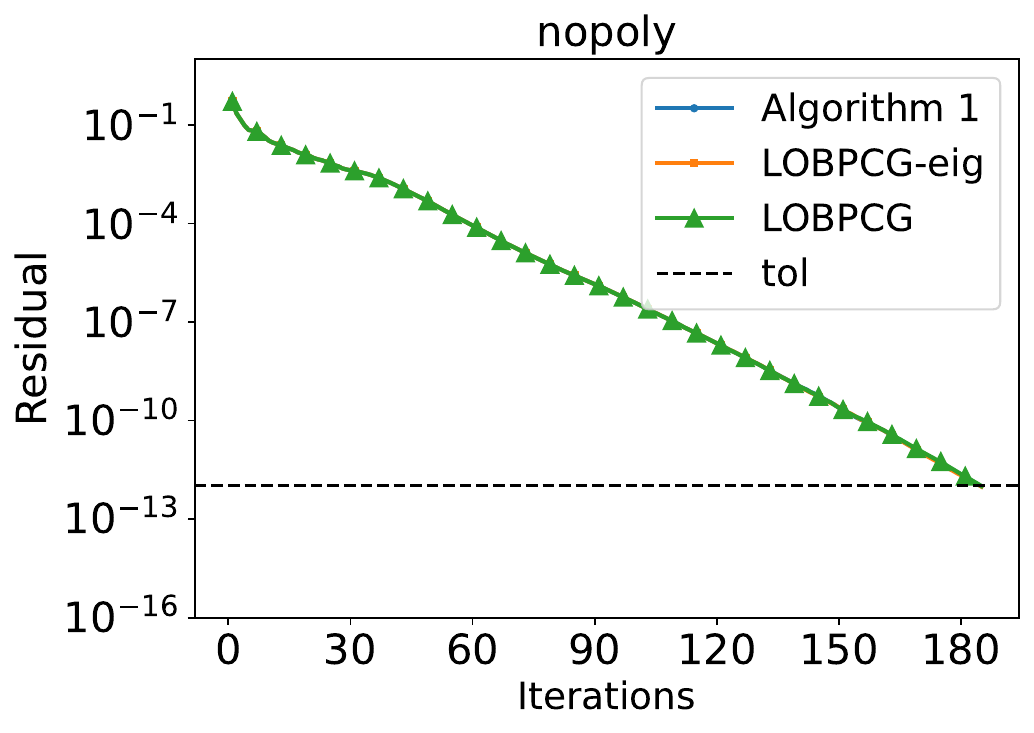} &
\includegraphics[width=0.3\textwidth]{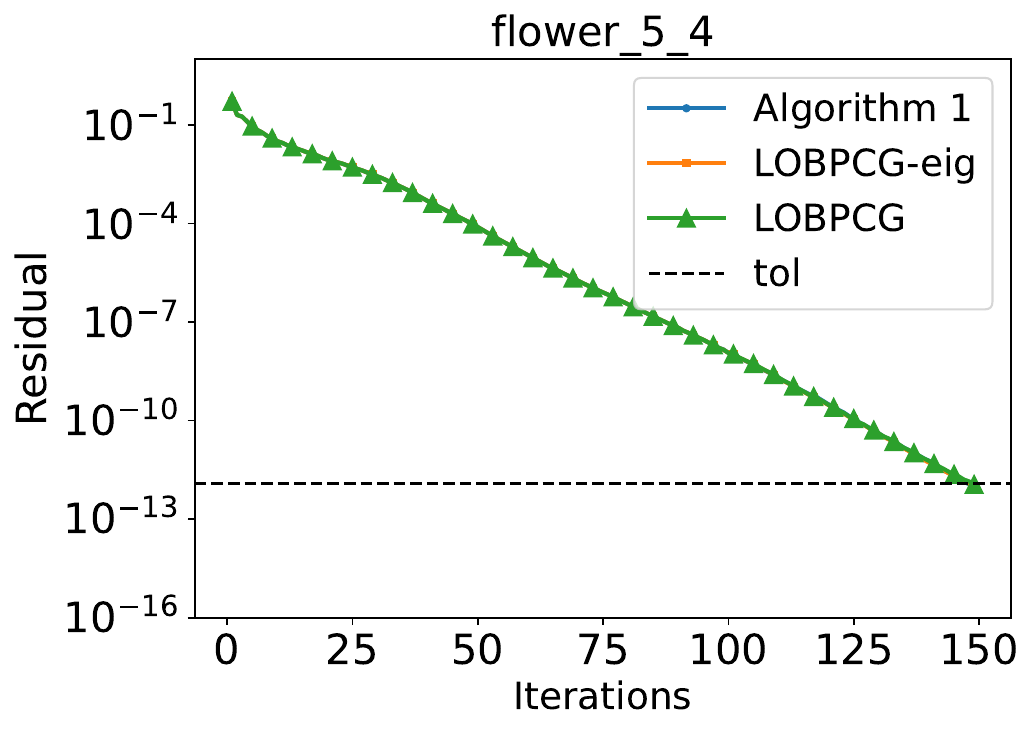} &
\includegraphics[width=0.3\textwidth]{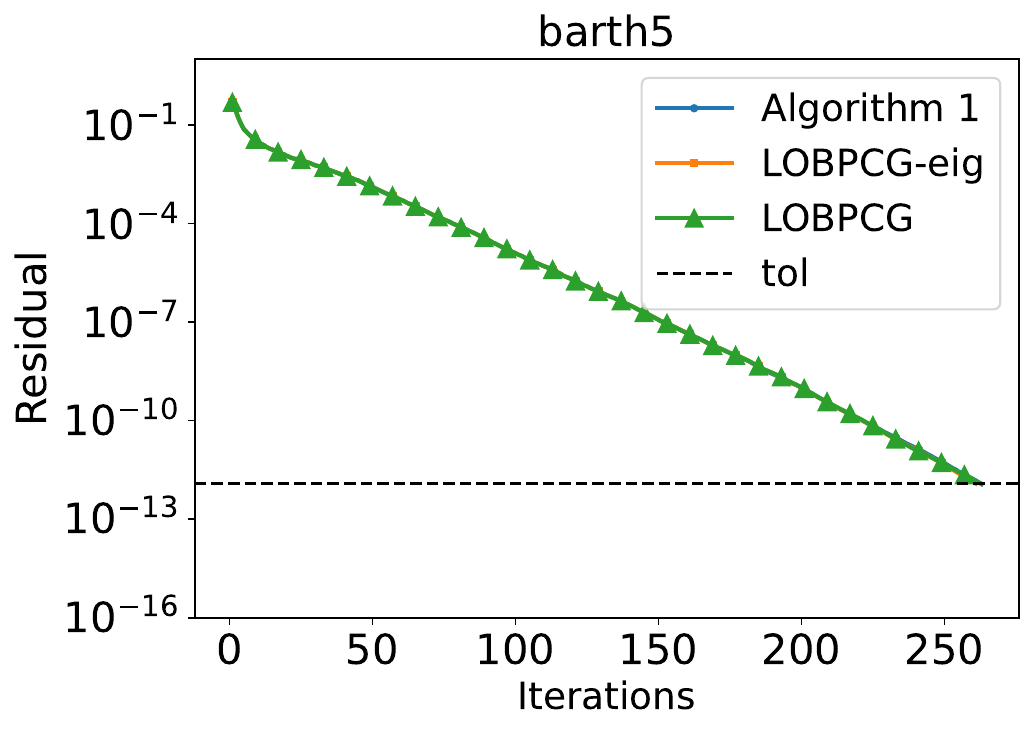} \\
\includegraphics[width=0.3\textwidth]{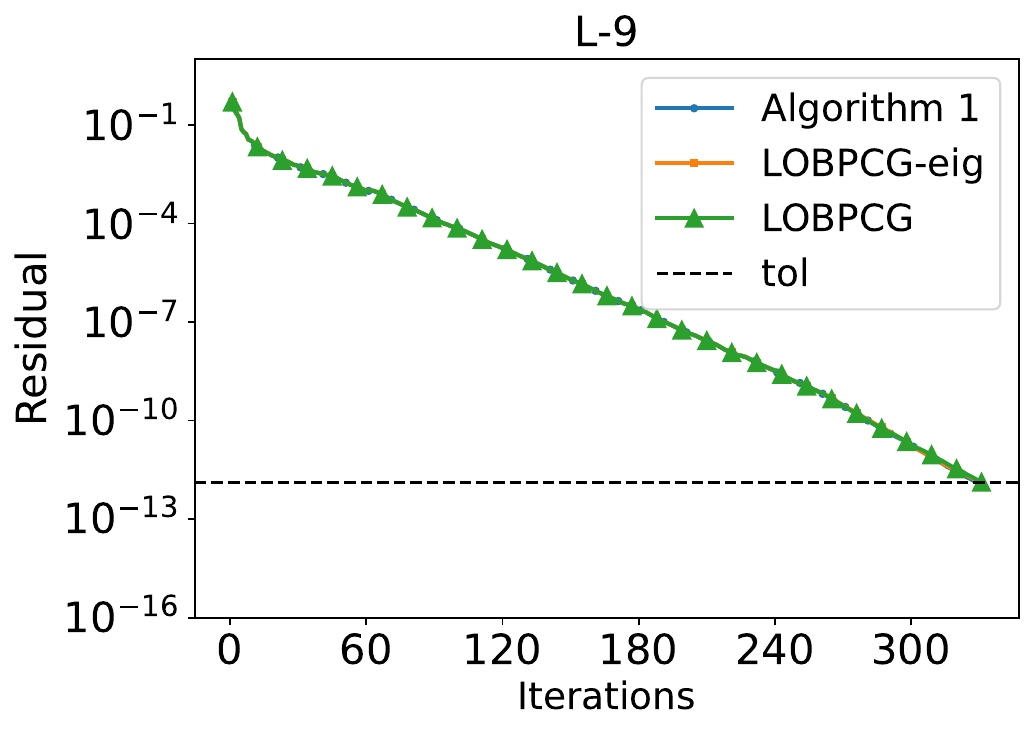} &
\includegraphics[width=0.3\textwidth]{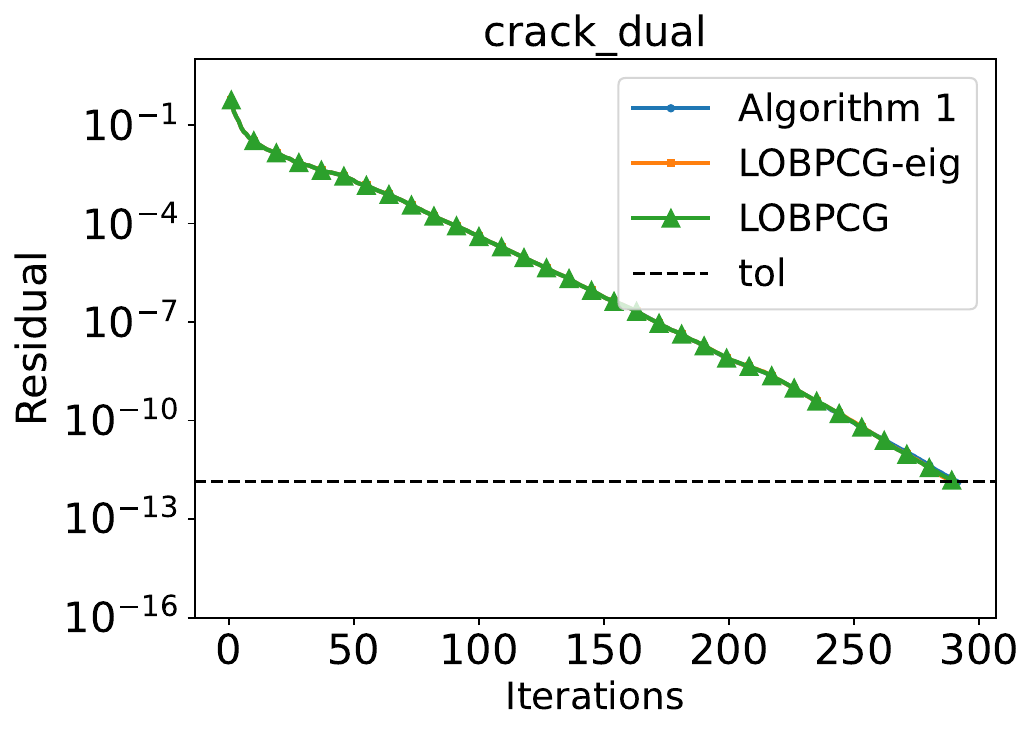} &
\includegraphics[width=0.3\textwidth]{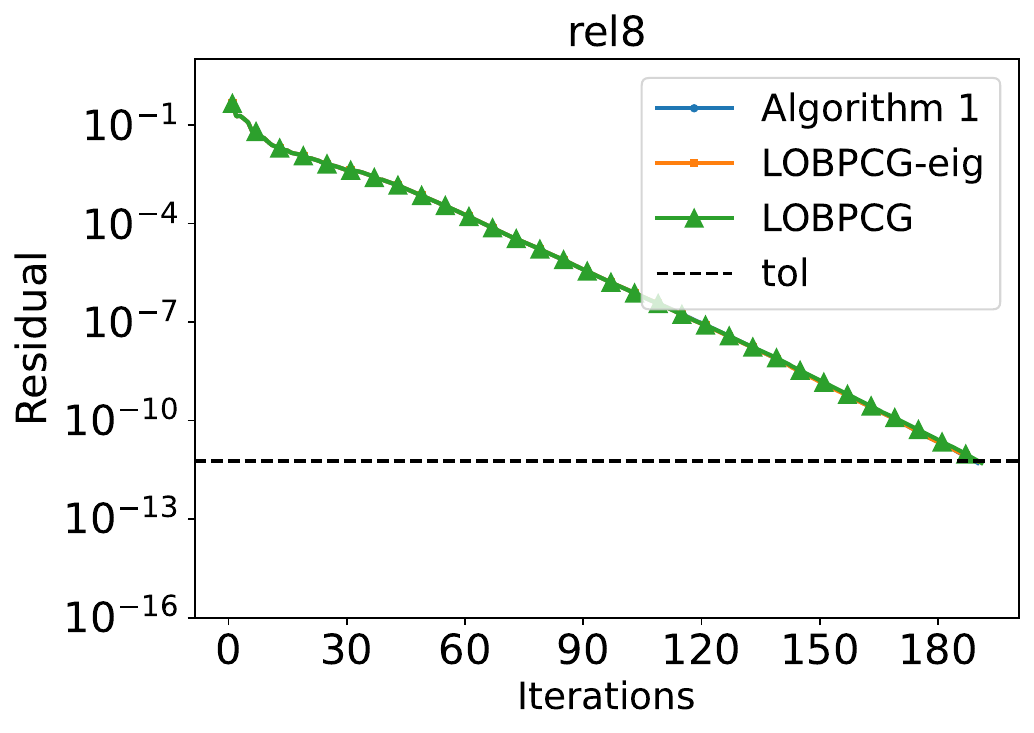}
\end{tabular}
\caption{Convergence histories for GSVD experiments with \(B=B_1\), without preconditioning.}
\label{fig:gsvdB1}
\end{figure}

\begin{table}[htbp]
\centering
\small
\setlength{\tabcolsep}{5pt}
\caption{CPU times (seconds) for the GSVD experiments with \(B=B_1\), without
preconditioning. A dagger indicates that the prescribed tolerance was not reached
within \(3000\) iterations.}
\label{tab:gsvdB1}
\begin{tabular}{lrrrrrr}
\toprule
ID & 1 & 2 & 3 & 4 & 5 & 6\\
\midrule
Algorithm~\ref{alg:gsvd}
& 1.373 & 3.476 & 3.345 & 12.22 & 209\(^{\dagger}\) & 4.893\\
LOBPCG--eig
& 1.757 & 4.476 & 4.198 & 14.58 & 250.4\(^{\dagger}\) & 5.758\\
LOBPCG
& 3.854 & 15.73 & 16.93 & 77.63 & 1507\(^{\dagger}\) & 34.79\\
\bottomrule
\end{tabular}

\medskip
\begin{tabular}{lrrrrrr}
\toprule
ID & 7 & 8 & 9 & 10 & 11 & 12\\
\midrule
Algorithm~\ref{alg:gsvd}
& 14.4 & 9.972 & 26.45 & 33.78 & 33.5 & 263.5\\
LOBPCG--eig
& 16.97 & 11.58 & 30.78 & 38.6 & 39.04 & 267.7\\
LOBPCG
& 126.2 & 77.52 & 268.9 & 362.2 & 351.3 & 2498\\
\bottomrule
\end{tabular}
\end{table}

\subsection{Preconditioned GSVD computation with \(B=B_2\)}
\label{sec:gsvdb2}

For the pair \((A,B_2)\), we compute the \(35\) largest generalized singular triplets.
For all three LOBPCG variants, the preconditioner is set to \(T_+=(\check A-\mu\check B)^{-1}\),
with the following adaptive shift rule.
Let \((\hat\sigma_i,\hat u_i,\hat w_i)\) be the \(i\)th largest approximate
generalized singular triplet.
The preconditioner is switched on once the relative residual of the largest 
approximate triplet falls below \(0.1\).
As long as the largest triplet has not converged, set \(\mu = 2\,\hat\sigma_1\); 
once the first \(i\) triplets have converged but the \((i+1)\)-st has not, set
\(\mu=\hat\sigma_i\).

\begin{table}[htbp]
\centering
\small 
\setlength{\tabcolsep}{5pt}
\caption{CPU times (seconds) for the GSVD experiments with \(B=B_2\),
using direct or inexact MINRES solves for preconditioning.}
\label{tab:pre_minres_direct}
\begin{tabular}{lrrrrrr}
\toprule
ID & 1 & 2 & 3 & 4 & 5 & 6\\
\midrule
\multicolumn{7}{c}{\textbf{Direct solves}}\\
\midrule
Algorithm~\ref{alg:gsvd}
& 2.976 & 2.747 & 5.461 & 10.37 & 21.59 & 9.953\\
LOBPCG--eig
& 3.405 & 2.927 & 5.587 & 10.6 & 22.61 & 10.1\\
LOBPCG
& 4.032 & 4.986 & 8.107 & 20.48 & 36.26 & 19.55\\
\midrule
\multicolumn{7}{c}{\textbf{Inexact MINRES solves}}\\
\midrule
Algorithm~\ref{alg:gsvd}
& 17.38 & 10.13 & 8.538 & 43.27 & 243.6 & 81.97\\
LOBPCG--eig
& 18.24 & 10.6 & 8.846 & 44.15 & 259.4 & 85.56\\
LOBPCG
& 21.28 & 12.84 & 11.88 & 67.63 & 380.4 & 132.3\\
\bottomrule
\end{tabular}

\medskip
\begin{tabular}{lrrrrrr}
\toprule
ID & 7 & 8 & 9 & 10 & 11 & 12\\
\midrule
\multicolumn{7}{c}{\textbf{Direct solves}}\\
\midrule
Algorithm~\ref{alg:gsvd}
& 20.04 & 52.27 & 23.79 & 20.38 & 29.9 & 1719\\
LOBPCG--eig
& 20.53 & 52.51 & 24.15 & 20.64 & 30.07 & 1701\\
LOBPCG
& 30.71 & 63.58 & 41.14 & 40.39 & 59.21 & 1925\\
\midrule
\multicolumn{7}{c}{\textbf{Inexact MINRES solves}}\\
\midrule
Algorithm~\ref{alg:gsvd}
& 93.22 & 31.69 & 135.6 & 131.6 & 194.7 & 1679\\
LOBPCG--eig
& 92.98 & 32.04 & 137.8 & 133.2 & 196.4 & 1701\\
LOBPCG
& 137.2 & 49.95 & 229.7 & 238.2 & 372.8 & 2713\\
\bottomrule
\end{tabular}
\end{table}

Table~\ref{tab:pre_minres_direct} reports the total runtimes of the experiments 
using either direct or inexact MINRES solves for preconditioning (stopping tolerance \(10^{-2}\), 
at most \(30\) iterations).
Each MINRES solve uses MATLAB's default zero initial vector.
Figure~\ref{fig:pre-minres} presents the corresponding
convergence histories for the MINRES-based implementation.
For these sparse test problems, the solution of the preconditioning systems accounts
for a substantial portion of the total computational cost. 
This common overhead reduces the relative benefit of the less expensive structured
Rayleigh--Ritz procedure in Algorithm~\ref{alg:gsvd}, compared with that observed in the
unpreconditioned experiments.
Nevertheless, Algorithm~\ref{alg:gsvd} remains consistently faster than the 
unstructured LOBPCG method and is generally competitive with, or faster than, 
the LOBPCG--eig variant.
The MINRES-based approximation preserves the overall convergence behavior
of the eigensolver and provides a practical alternative when sparse
direct factorization is prohibitively expensive in terms of computation
or memory consumption.
It is, however, not uniformly faster than a direct solve in the present experiments, 
since its efficiency depends on the convergence of the inner iterations. 

\begin{figure}[htbp]
\centering
\begin{tabular}{ccc}
\includegraphics[width=0.3\textwidth]{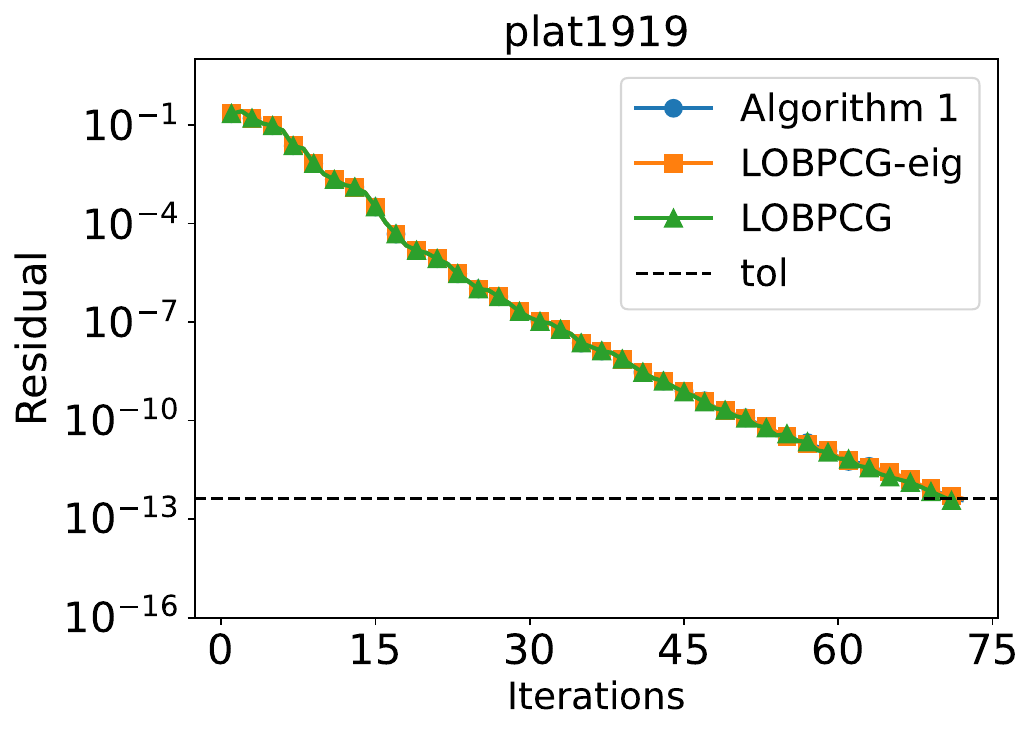} &
\includegraphics[width=0.3\textwidth]{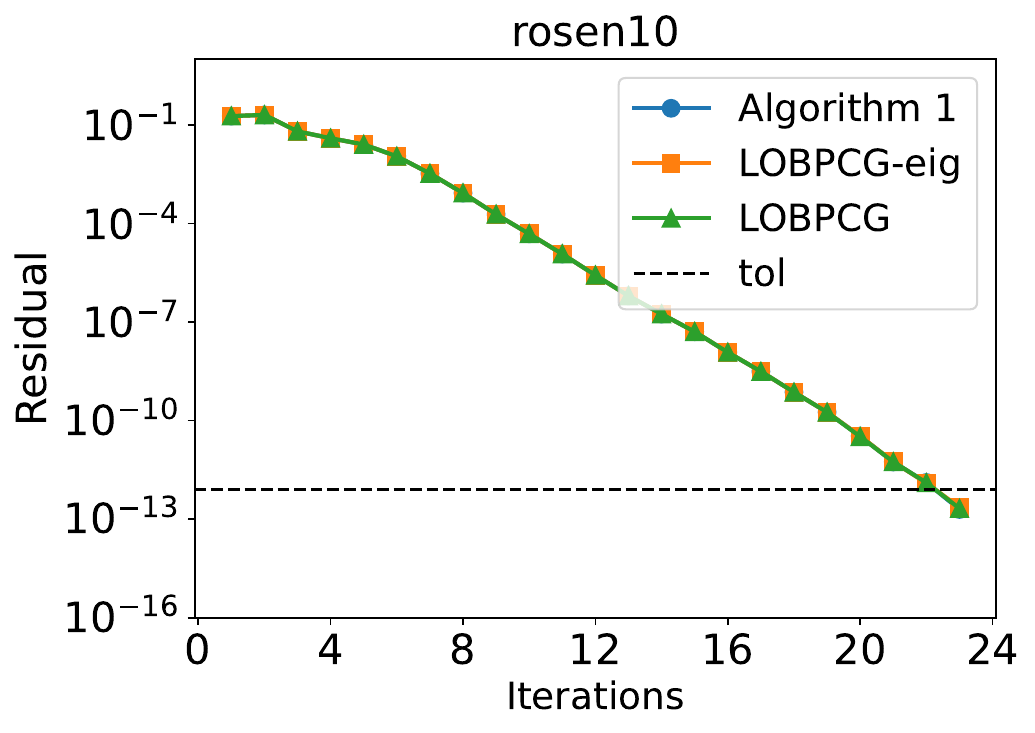} &
\includegraphics[width=0.3\textwidth]{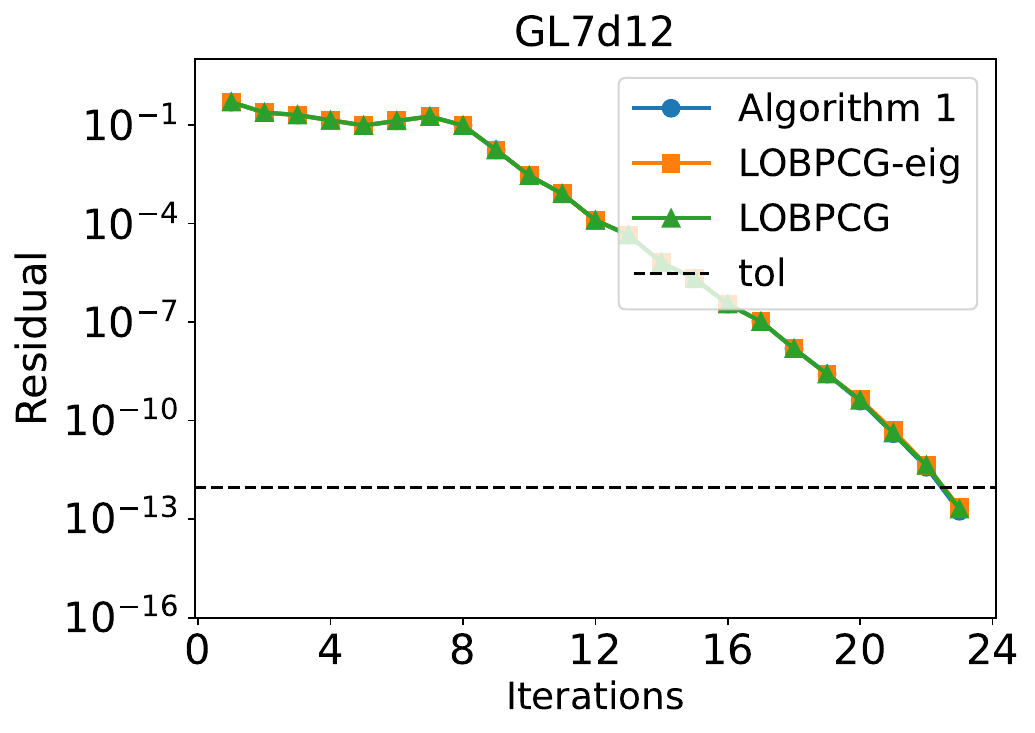} \\
\includegraphics[width=0.3\textwidth]{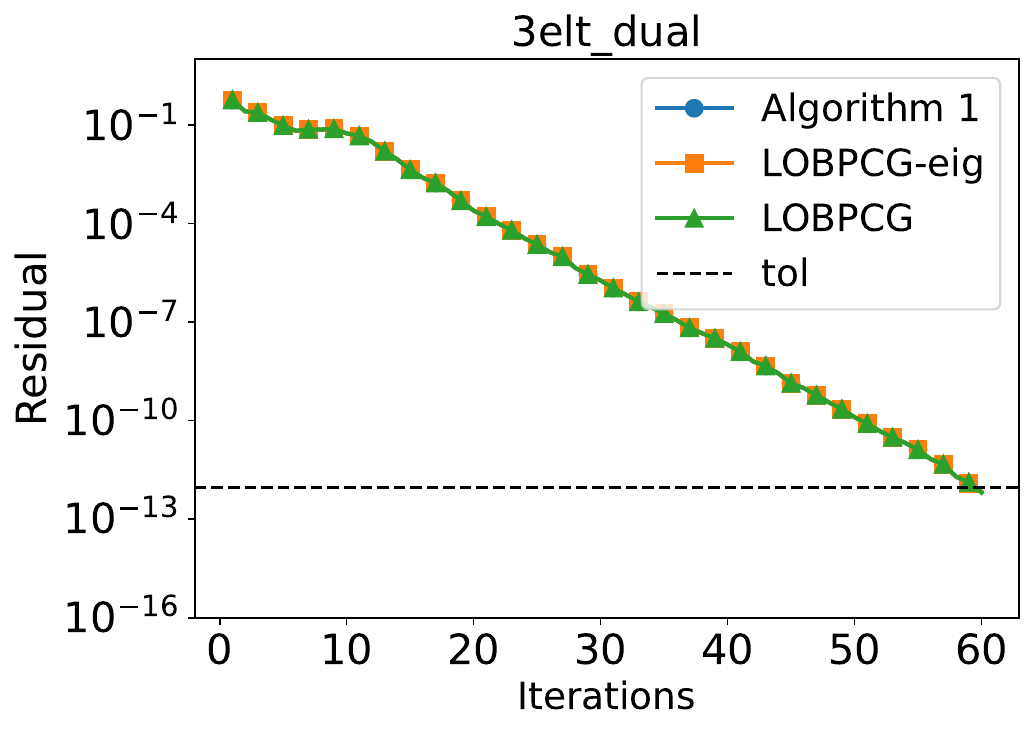} &
\includegraphics[width=0.3\textwidth]{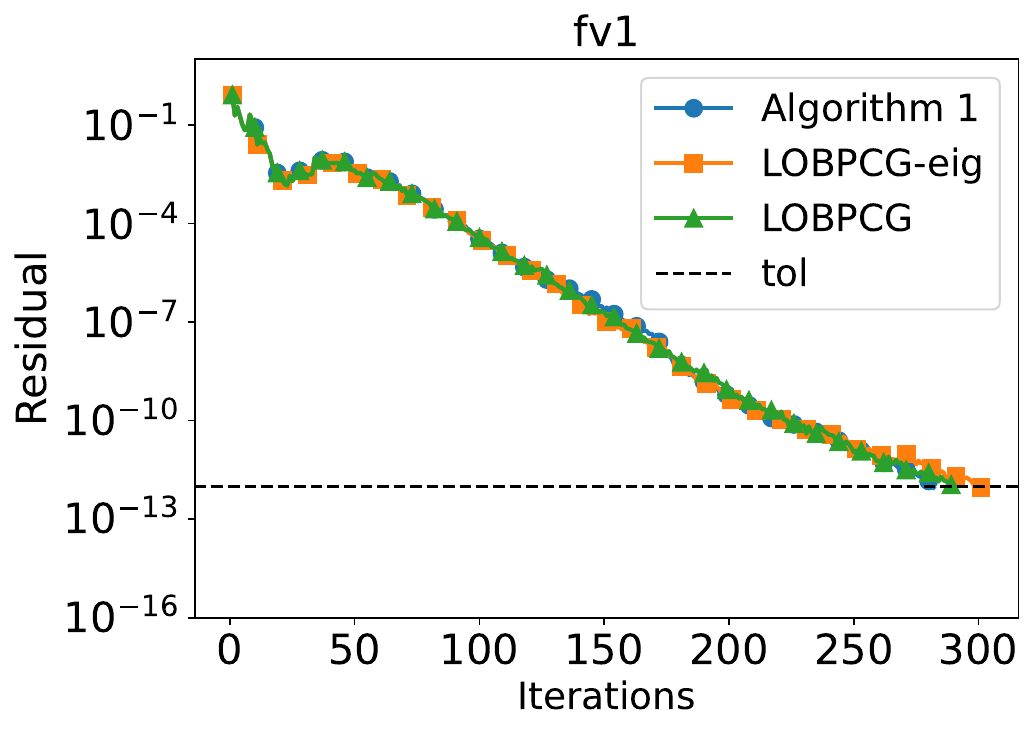} &
\includegraphics[width=0.3\textwidth]{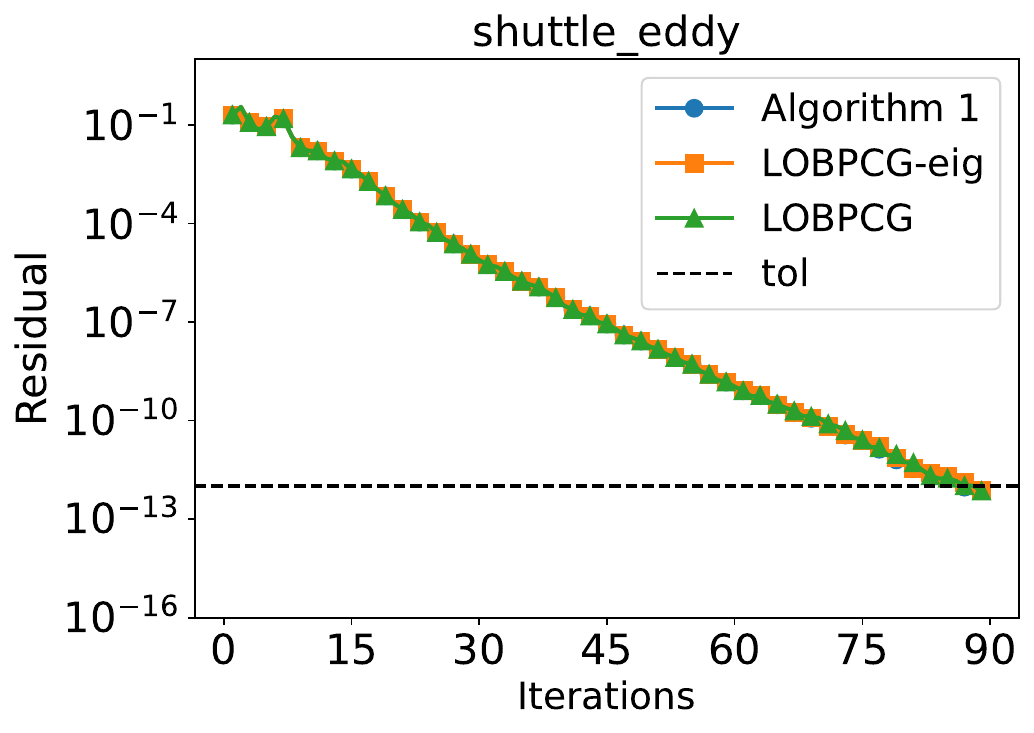} \\
\includegraphics[width=0.3\textwidth]{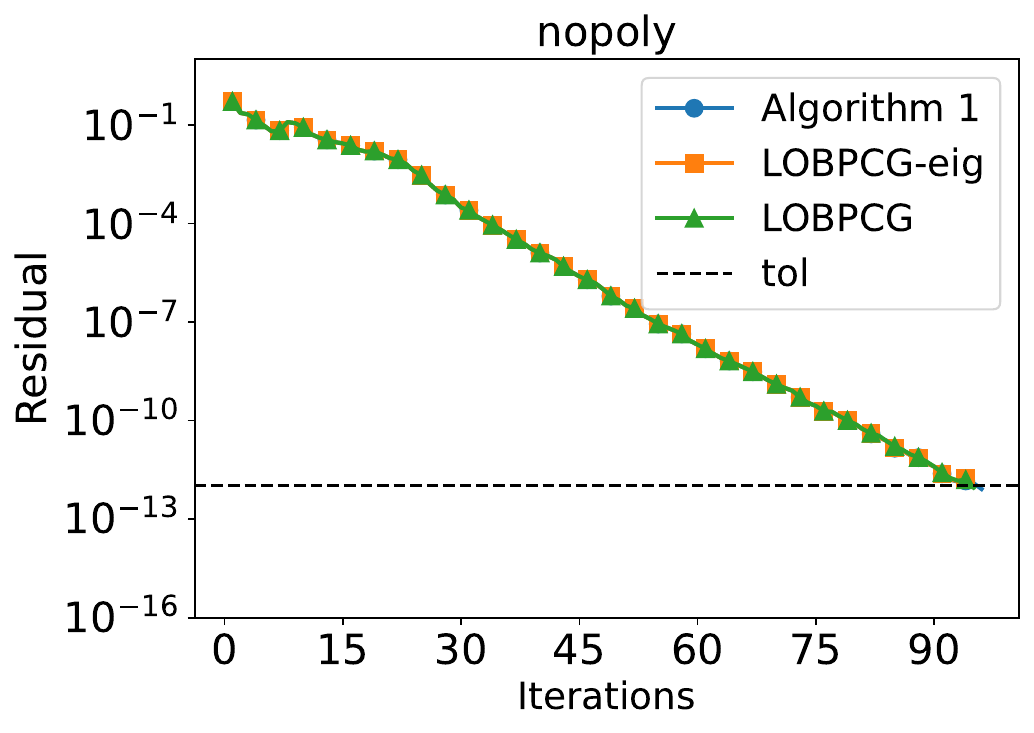} &
\includegraphics[width=0.3\textwidth]{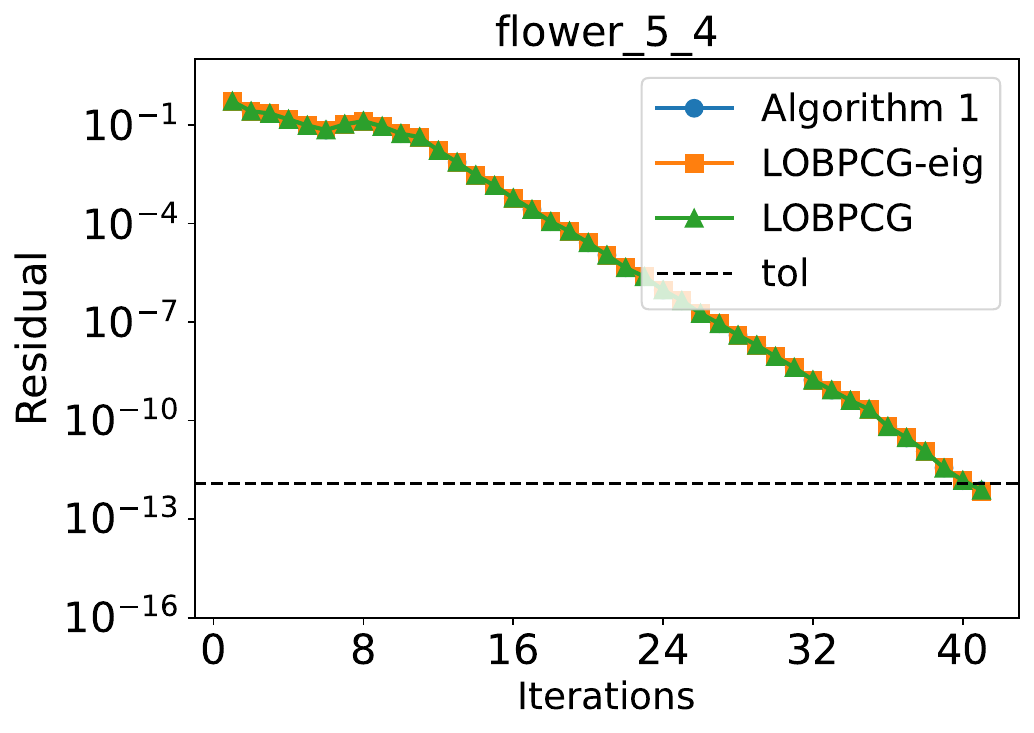} &
\includegraphics[width=0.3\textwidth]{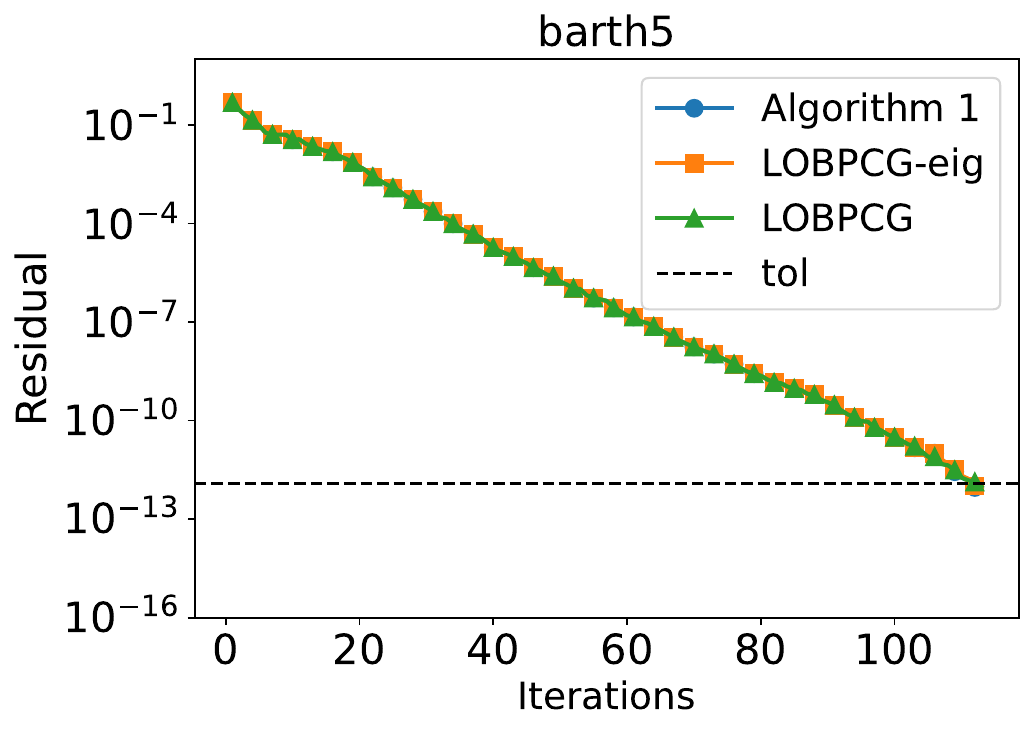} \\
\includegraphics[width=0.3\textwidth]{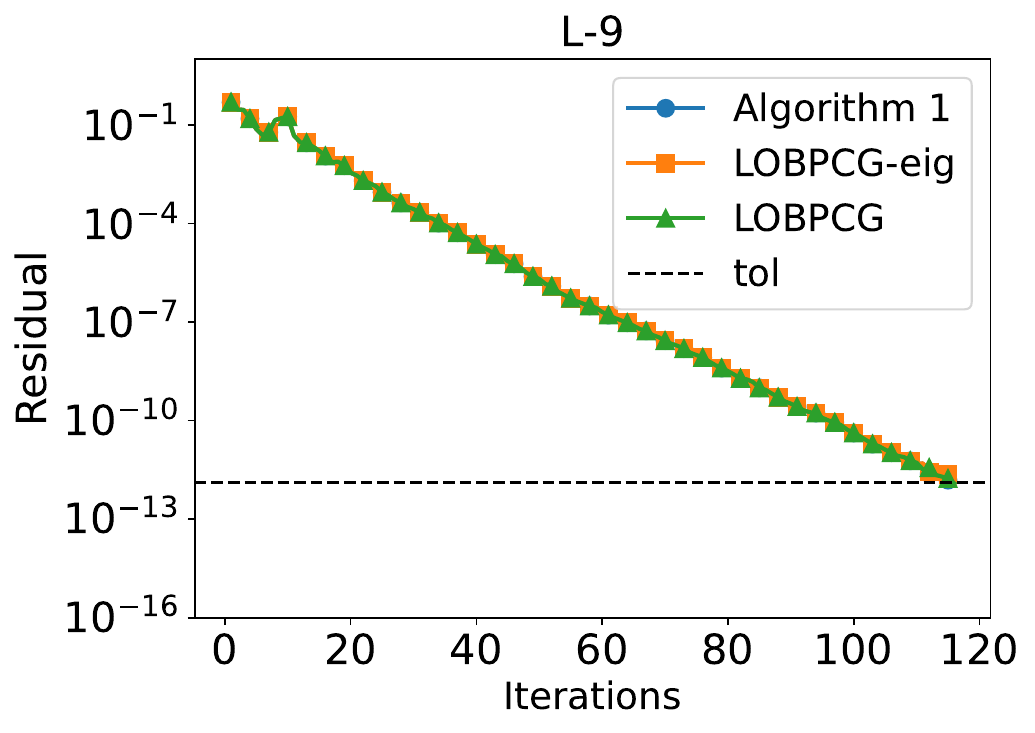} &
\includegraphics[width=0.3\textwidth]{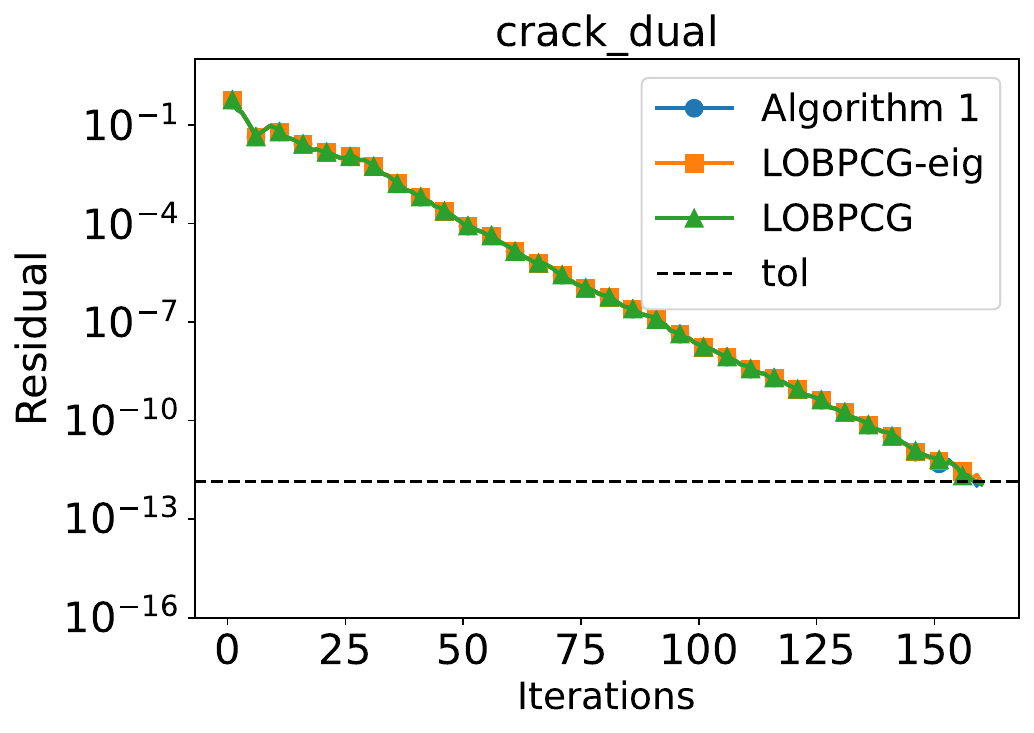} &
\includegraphics[width=0.3\textwidth]{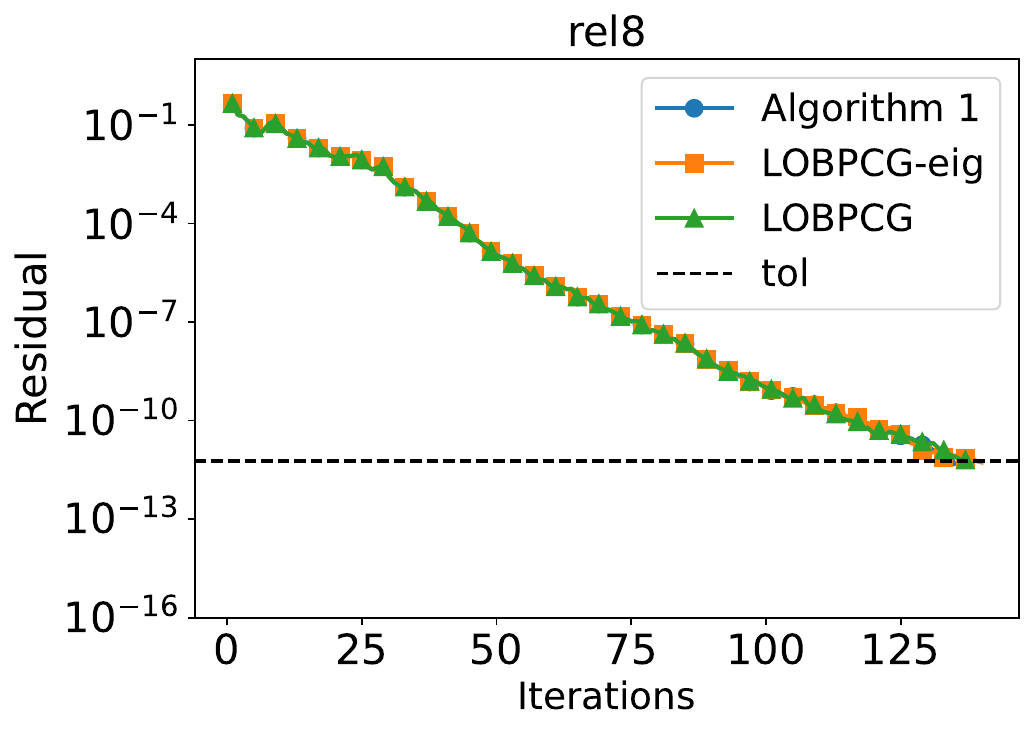}
\end{tabular}
\caption{Convergence histories for the GSVD experiments with \(B=B_2\), 
using inexact MINRES solves for preconditioning.}
\label{fig:pre-minres}
\end{figure}

\subsection{Comparison with the Gram-based solver}
\label{sec:gram}
In addition, we compare our algorithm with a Gram-matrix-based algorithm.
Recall that the squared generalized singular values of \((A,B)\) coincide
with the generalized eigenvalues of the Hermitian-definite Gram-type pencil
\[
(A\herm A)x=\lambda(B\herm B)x,
\qquad \lambda=\sigma^2.
\]
Consequently, the \(l\) largest generalized singular values can alternatively
be computed by applying LOBPCG in the \(B\herm B\) inner product, equipped with
the IHL trick, to the matrix pair \((A\herm A,B\herm B)\), and then recovering
\(\sigma_i=\sqrt{\lambda_i}\) from the \(l\) largest generalized eigenvalues.

Let \(n=1000\), and let \(\hat U\), \(\hat V\in\mathbb{R}^{n\times n}\) be random 
orthogonal matrices. 
For a prescribed sequence \(\{\sigma_i\}_{i=1}^n\), define
\[
\Sigma=\operatorname{diag}(\sigma_1,\ldots,\sigma_n),
\qquad
A=\hat U\Sigma\hat V^{T}B_1,
\qquad
B=B_1.
\]
The generalized singular values of the matrix pair \((A,B_1)\) are exactly \(\sigma_1,\ldots,\sigma_n\). 
This construction provides a genuine GSVD problem with a nonidentity constraint matrix
while allowing the exact generalized singular values to be prescribed explicitly. 
We consider the following two cases.

\textbf{Case 1} (\(l=5\)). The generalized singular values are prescribed as
\[
\sigma_i
=
10^{\,2-\frac{7}{4}(i-1)},
\qquad i=1,\ldots,5,\qquad
\sigma_i=10^{-8},
\qquad i>5.
\]
Thus, the five desired generalized singular values range from \(\sigma_1=10^2\) to
\(\sigma_5=10^{-5}\), spanning seven orders of magnitude.
In the Gram formulation, these values are squared, so the corresponding generalized eigenvalues 
range from \(10^4\) to \(10^{-10}\), whereas the unwanted eigenvalues are of order \(10^{-16}\). 
We compute the five largest generalized singular values using a block size of eight.

\textbf{Case 2} (\(l=50\)).
The first five generalized singular values are prescribed by
\[
\sigma_i
=
10^{\,2-\frac{1}{4}(i-1)},
\quad i=1,\ldots,5,\quad
\sigma_i=
10^{-4-\frac{i-6}{44}},
\quad i=6,\ldots,50,\quad
\sigma_i=10^{-7},
\quad i>50.
\]
Consequently, the target set contains five large generalized singular values in
\([10,10^2]\) and forty-five much smaller values in \([10^{-5},10^{-4}]\). 
After passing to the Gram pencil \((A\trans A,B_1\trans B_1)\), the two groups lie 
in \([10^2,10^4]\) and \([10^{-10},10^{-8}]\), respectively, while the unwanted 
eigenvalues are of order \(10^{-14}\). 
We compute the fifty largest generalized singular values using a block size of sixty.

For both cases, the stopping tolerance is \(10^{-14}\sqrt{n}\), the maximum number 
of updates is \(50\), and no preconditioner is employed. 
Both methods use the same starting right subspace, the same IHL update, and the same 
deflation and stopping criteria.
For the Gram-based method, positive Ritz pairs can be converted to triplets
using \(\hat\sigma_i=\sqrt{\hat\lambda_i}\) and \(\hat u_i=A\hat w_i/\hat\sigma_i\), with
\(\|B\hat w_i\|_2=1\); the two triplet residuals should then be
evaluated in the original equations. 
The relative error of the \(i\)th value is \(|\hat\sigma_i-\sigma_i|/\sigma_i\),
and the maximum is taken over \(1\le i\le l\).

The numerical behavior is shown in Figure~\ref{fig:bad_case}.
In Case \(1\), Algorithm~\ref{alg:gsvd} converges in \(32\) iterations with a residual 
of \(1.817\times10^{-13}\), whereas the Gram-based method fails to reach the prescribed
tolerance within \(50\) iterations and attains only \(1.061\times10^{-10}\). 
Their maximum relative errors in the computed generalized singular values are \(1.128\times10^{-11}\) 
and \(1.083\times10^{-5}\), respectively. 
In Case \(2\), Algorithm~\ref{alg:gsvd} converges in \(31\) iterations with a residual of
\(1.686\times10^{-13}\), while the Gram-based method again reaches the iteration limit with
a residual of \(5.264\times10^{-9}\). 
The corresponding maximum relative errors are \(1.681\times10^{-11}\) and \(1.732\times10^{-2}\). 
These results illustrate the loss of relative accuracy caused by explicitly forming the 
Gram matrices, particularly when very large and very small generalized singular values
belong to the same target set.
\begin{figure}
\centering
\includegraphics[width=0.9\linewidth]{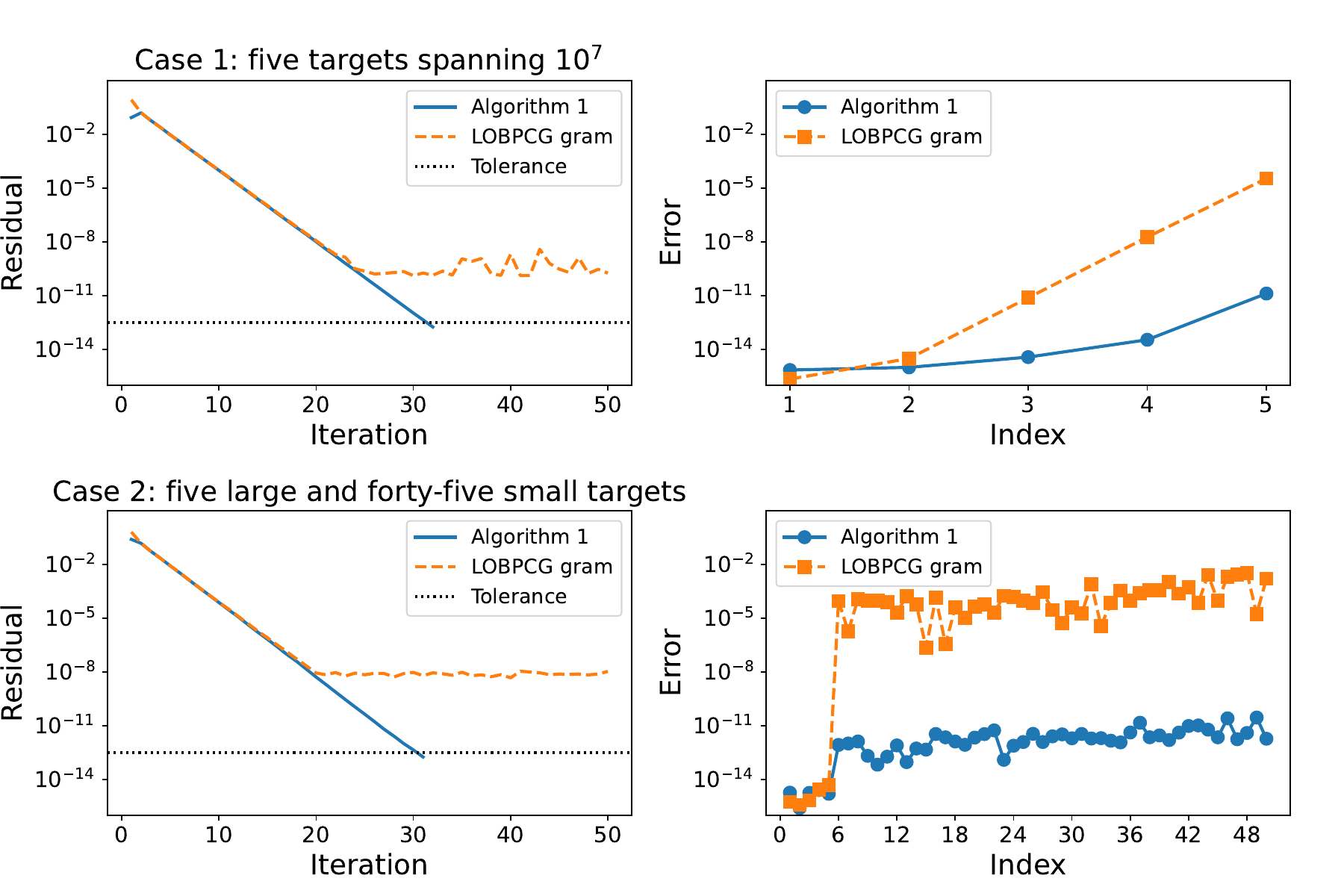}
\caption{Convergence and accuracy comparison of Algorithm~\ref{alg:gsvd} and the Gram-based
        LOBPCG method for the GSVD with \(B=B_1\).
        The top and bottom rows correspond to Cases~\(1\) and~\(2\), respectively.
        The plots in the left column show the convergence histories measured by
        the relative residuals, while those in the right column show the
        relative errors in the computed generalized singular values.
        In the right-column plots, index \(j\) denotes the \(j\)th largest desired
        generalized singular value \(\sigma_j\).}
\label{fig:bad_case}
\end{figure}

\section{Conclusions}
\label{sec:concl}
In this paper, we developed a structure-preserving LOBPCG algorithm for
computing several of the largest nontrivial generalized singular triplets of a
large matrix pair \((A,B)\), where \(B\) has full column rank. 
By exploiting the positive-negative spectral symmetry of the associated 
Hermitian-definite Jordan--Wielandt pencil, the method represents the two spectral branches 
through paired left and right component spaces.
When both component bases have full dimension, the resulting structured Galerkin condition
replaces the Rayleigh--Ritz procedure on a \(6k\)-dimensional augmented space by 
a complete SVD of the \(3k\times3k\) projected cross matrix \(S_U\herm AS_W\).
The algorithm constructs its conjugate search directions by applying IHL to the
left and right SVD coefficients, retaining the old approximation in the
next search space.
The established coefficient correspondence provides a simple way to construct
the IHL update for the augmented LOBPCG algorithm from the projected
eigenvector coefficients.

The present method assumes that \(B\) has full column rank and targets 
the largest nontrivial generalized singular values. 
Its convergence can deteriorate when the desired values are tightly 
clustered with the remaining spectrum, as observed for one of the SuiteSparse 
test problems. 
Future work will therefore consider adaptive block-size and locking strategies for
clustered spectra and extensions to rank-deficient matrix pairs and other parts of
the GSVD spectrum.

\section*{Acknowledgments}
The author thanks Meiyue Shao for helpful discussions.

\addcontentsline{toc}{section}{References}
\bibliographystyle{plainurl}
\bibliography{GSVD-LOBPCG}

\end{document}